\documentclass[a4paper,11pt,reqno, english]{amsart}

\usepackage{anysize,color}
\usepackage[utf8]{inputenc}
\usepackage[dvipsnames]{xcolor}
\usepackage[colorlinks=true,urlcolor=blue,linkcolor={YellowOrange},citecolor={YellowOrange}]{hyperref}  
\usepackage{float}
\usepackage{amsfonts,amssymb,amsmath}
\usepackage{amsthm}    
\usepackage{url}
\usepackage{paralist}
\usepackage{tikz}
    \usetikzlibrary{decorations.markings,arrows,shapes.callouts}
\usepackage{rotating} 
\usepackage[mathscr]{euscript}
\usepackage{verbatim}
\usepackage{tikz-cd}
\usepackage{mathtools}
\usepackage{dsfont}
\usepackage{amsrefs}
\usepackage{thmtools}
\usepackage{thm-restate}

\usepackage{subfig}

\usepackage{charter}

\newtheorem{theorem}{Theorem}[section]

\newtheorem{lemma}[theorem]{Lemma} 
\newtheorem{conjecture}[theorem]{Conjecture}
\newtheorem{corollary}[theorem]{Corollary} 

\theoremstyle{definition}

\theoremstyle{remark}

\newtheorem{remark}[theorem]{Remark} 

\newcommand{\R}{\mathds{R}}
\newcommand{\rr}{\mathds{R}}
\newcommand{\C}{\mathbb{C}}
\newcommand{\CP}{\mathbb{C}\mathrm{P}}
\newcommand{\RP}{\mathds{R}\mathrm{P}}
\newcommand{\PP}{\mathbb{P}}

\newcommand{\M}{\mathcal{M}}
\newcommand{\Z}{\mathbb{Z}}

\newcommand{\F}{\mathbb{F}}

\newcommand{\Sym}{\mathfrak{S}}

\newcommand{\complexModule}[1]{\C_{#1}}
\newcommand{\grassmannC}[2]{G_{#1}(\C^{#2})}
\newcommand{\grassmannR}[2]{G_{#1}(\rr^{#2})}
\newcommand{\FHindex}[3]{\mathrm{ind}_{#1}^{#2}(#3)}
\newcommand{\FHindexBig}[3]{\mathrm{ind}_{#1}^{#2}\Big(#3\Big)}
\newcommand{\deletedJoin}[2]{(#1)^{* #2}_{\Delta}}
\newcommand{\deletedJoinNoBrackets}[2]{#1^{* #2}_{\Delta}}

\def\keywords{\xdef\@thefnmark{}\@footnotetext}

\DeclareMathOperator{\id}{\mathrm{id}}
\DeclareMathOperator{\pt}{\mathrm{pt}}

\DeclareMathOperator{\supp}{\mathrm{supp}}

\DeclareMathOperator{\conv}{\mathrm{conv}}
\DeclareMathOperator{\codim}{\mathrm{codim}}

\DeclareMathOperator{\linspan}{\mathrm{span}}
\DeclareMathOperator{\spanR}{\mathrm{span}_{\R}}
\DeclareMathOperator{\spanC}{\mathrm{span}_{\C}}
\DeclareMathOperator{\rank}{\mathrm{rk}}

\begin{document}

\title{Complex analogues of the Tverberg--Vre\'cica conjecture and central transversal theorems}


\author[Sadovek]{Nikola Sadovek} 
\address{Institute of Mathematics, Freie Universität Berlin, Germany}
\email{nikolasdvk@gmail.com,~nikola.sadovek@fu-berlin.de}

\author[Sober\'on]{Pablo Sober\'on} 
\address{Baruch College and The Graduate Center, City University of New York, New York, USA.}
\email{psoberon@gc.cuny.edu}

\thanks{The research of N. Sadovek is funded by the Deutsche Forschungsgemeinschaft (DFG, German Research Foundation) under Germany's Excellence Strategy – The Berlin Mathematics Research Center MATH+ (EXC-2046/1, project ID 390685689, BMS Stipend). The research of P. Sober\'on was supported by NSF CAREER award no. 2237324 and a PSC-CUNY Trad B award.}


\begin{abstract}
The Tverberg--Vre\'cica conjecture claims a broad generalization of Tverberg's classical theorem.  One of its consequences, the central transversal theorem, extends both the centerpoint theorem and the ham sandwich theorem.  In this manuscript, we establish complex analogues of these results, where the corresponding transversals are complex affine spaces. The proofs of the complex Tverberg--Vre\'cica conjecture and its optimal colorful version rely on the non-vanishing of an equivariant Euler class. Furthermore, we obtain new Borsuk--Ulam-type theorems on complex Stiefel manifolds. These theorems yield complex analogues of recent extensions of the ham sandwich theorem for mass assignments by Axelrod-Freed and Sober\'on, and provide a direct proof of the complex central transversal theorem.
\end{abstract}

\maketitle


\section{Introduction}

Rado's centerpoint theorem \cite{Rado:1946ud} and Steinhaus' ham sandwhich theorem \cite{Steinhaus1938} are cornerstone results in discrete and computational geometry. 
The methods used to prove these theorems, along with their numerous extensions, have had a great impact in the development of topological combinatorics \cites{matouvsek2003borsuk, RoldanPensado2022, Holmsen:2017uf}. In particular, a remarkable common generalization of the ham sandwich theorem and the centerpoint theorem -- the central transversal theorem -- was proven independently by \v{Z}ivaljevi\'c and Vre\'cica and by Dolnikov \cites{Zivaljevic1990, Dolnikov:1992ut}.

\begin{theorem}[\v{Z}ivaljevi\'c and Vre\'cica 1990, Dolnikov 1992]
    Let $0 \le k < d$ be integers.  Given $k+1$ probability measures $\mu_0, \dots, \mu_{k}$ on $\rr^d$, there exists a $k$-dimensional affine subspace $V$ such that every half-space $H$ that contains $V$ satisfies
    \[
    \mu_i (H) \ge \frac{1}{d-k+1}
    \]
    for each $i = 0, \dots, k$.
\end{theorem}

The case $k=0$ is Rado's centerpoint theorem, while the case $k=d-1$ is the ham sandwich theorem.  The value $1/(d-k+1)$ is referred to as the ``depth'' of the transversal.  Since the ham sandwich theorem is a consequence of the Borsuk--Ulam theorem, we can expect the topological methods needed to prove its generalizations to also increase in complexity.  Indeed, the earliest proofs of the central transversal theorem relied on explicit computations of characteristic classes of some associated vector bundles, or the computation of the Fadell--Husseini index of some $(\Z_2)^k$-equivariant spaces.  However, simpler proofs have been discovered since then, using only Borsuk--Ulam type theorems for Stiefel manifolds \cite{Manta2024}.

The goal of this manuscript is to establish new analogues of the Tverberg--Vre\'cica conjecture (see Conjecture \ref{conj:Tverberg-Vrecica}), including an ``optimal colorful" version in the sense of Blagojevi\'c--Matschke--Ziegler \cite{blagojevic2011optimal}, as well as the central transversal theorem.  In our results, the transversals are required to be complex spaces. Our motivation is partially driven by recent advances in piercing problems for families of convex sets using flats. One classical example is the Goodman--Pollack characterization of families of convex sets in $\rr^d$ with a real hyperplane transversal \cite{Goodman1988}.  This has been shown by Holsmen \cite{Holmsen2022} to have a colorful generalization. McGinnis \cites{McGinnis2023, mcginnis2023complex} showed that the Goodman--Pollack characterization has a complex extension as well, by providing necessary and sufficient conditions for existence of a complex hyperplane transversal for a family of convex sets in $\C^d$.

The combinatorial structure of $k$-dimensional affine spaces of $\C^d$ is interesting.  By identifying $\C^d$ with $\rr^{2d}$, every complex $k$-dimensional affine space corresponds to a real $2k$-dimensional affine space, but the converse is not true.  Therefore, in a complex analogue for complex $k$-flats in $\C^d$ of the central transversal theorem a key question is whether it behaves like the real case for $k$-flats in $\rr^d$ or like the real case for $2k$-flats in $\rr^{2d}$.  We obtain a mixed version, where the number of measures corresponds to case for $k$-flats in $\rr^d$ but the depth of the transversal corresponds to what we would expect for $2k$-flats in $\rr^{2d}$.  As a hyperplane in $\C^d$ does not split $\C^d$ into two connected components, we use real halfspaces to measure the depth of a flat.

\begin{restatable}[Complex central transversal]{theorem}{thmCentral}
\label{thm:complex-central-transversal}
	Let $0 \le k < d$ be integers and let $\mu_0, \dots, \mu_{k}$ be probability measures in $\C^d$, which are absolutely continuous with respect to the Lebesgue measure. Then, there exists a $k$-dimensional affine complex subspace $V \subseteq \C^d$ such that for every closed affine real halfspace $H \subseteq \C^d$ containing $V$ we have 
    \[
        \mu_i(H) \ge \frac{1}{2d-2k+1}
    \]
    for each $i = 0, \dots, k$. 
\end{restatable}

In Section \ref{sec:central transversal} we also show that both the number of measures $k+1$ and the depth $1/(2d-2k+1)$ are optimal.  The case $k=0$ of Theorem \ref{thm:complex-central-transversal} recovers the (real) centerpoint theorem in even dimensions.  The case $k=d-1$ yields the following complex version of the ham sandwich theorem.

\begin{restatable}[Complex Ham sandwich]{corollary}{corHam} \label{cor: complex ham sandwich}
	Let $d \ge 1$ be an integer and $\mu_0, \dots, \mu_{d-1}$ probability measures in $\C^d$, which are absolutely continuous with respect to the Lebesgue measure. Then, there exists an affine complex hyperplane $V$ such that for for every closed affine real halfspace $H \subseteq \C^d$ containing $V$ we have $\mu_i(H) \ge \frac{1}{3}$ for each $i = 0, \dots, d-1$. 
\end{restatable}

The structure of the first proof of Theorem \ref{thm:complex-central-transversal}, and its extension Theorem \ref{thm:central-for-flags}, follows the structure of the proof of Manta and Sober\'on of the central transversal theorem \cite{Manta2024}.  This requires a complex analogue, Theorem \ref{thm:equiv} below, of the Borsuk--Ulam-type theorem for real Stiefel manifolds by Chan, Chen, Frick, and Hull \cite{ChanChenFrickHull20}, which is of independent interest \cite{mcginnis2024necessary}.  This result is slightly weaker then the $\Z_2^n$-version in Remark \ref{rem:Z_2^n}, but showcases a similarity with the main result of Chan et al.

\begin{restatable}{theorem}{thmEquiv}
\label{thm:equiv}
    Let $d \ge n \ge 1$ be integers. Let $W_n(\C^d)$ be the complex Stiefel manifold of orthonormal $n$-frames in $\C^d$.  Then, every continuous $(S^1)^n$-equivariant map
    \begin{equation*}
        f:W_n(\C^d) \longrightarrow \C^{d-1}\oplus \C^{d-2} \oplus  \dots \oplus \C^{d-n}
    \end{equation*}
    hits the origin, where $(S^1)^n$ has the product action on both spaces.
\end{restatable}

We showcase a second application of Theorem \ref{thm:equiv} by proving a complex analogue of the ``fairy bread sandwich theorem'' of Axelrod-Freed and Sober\'on \cite{AxelrodFreed2022}, Theorem \ref{thm:complex-fairy-bread}.  The fairy bread sandwich theorem is an extension of the ham sandwich theorem for mass assignments.  A mass assignment for $k$-flats in $\rr^d$ is continuous assignment of probability measures to each $k$-dimensional (linear or affine) space of $\rr^d$. We will moreover call it nice if it is absolutely continuous with respect to the Lebesgue measure on each $k$-flat. Recently, the study of mass assignments has gathered interest since many mass partition results have a stronger mass assignment version \cites{Schnider:2020kk, BlagojevicCrabb23, Camarena2024, Blagojevic2023a}.  Our results show that these improvements also carry through to complex analogues.

\begin{restatable}[Complex fairy bread sandwich]{theorem}{thmFairy}\label{thm:complex-fairy-bread}
    Let $d \ge k \ge 1$ be integers. Suppose that $(\pi_{k-1}, \dots, \pi_{d-1})$ is a permutation of $(k-1, \dots, d-1)$ and assume that for each $i=k-1, \dots ,d-1$ we have nice mass assignments $\mu^i_0,\mu^i_{1}, \dots , \mu^i_{\pi_i}$ on the $(i+1)$-dimensional complex affine subspaces of $\C^d$.
    Then, there exists a flag
    \[
        S_{k-1} \subseteq \dots \subseteq S_{d-1} \subseteq S_d=\C^d
    \]
    of affine complex subspaces of dimensions $k-1, \dots, d-1$, respectively, such that each $S_{i}$ is a complex central transversal for the measures $\mu^i_{0}[S_{i+1}], \dots , \mu^i_{\pi_i}[S_{i+1}]$ in $S_{i+1}$. Concretely, we have
    \[
        \mu_0^i[S_{i+1}](H), \dots, \mu_{\pi_i}^i[S_{i+1}](H) \ge \frac{1}{3},
    \]
    for any (real) halfspace $H$ in $S_{i+1}$ containing $S_i$.
\end{restatable}

Another pillar of combinatorial geometry is Tverberg's theorem \cite{Tverberg:1966tb}.  Proven in 1966, it is a fundamental result at the crossroads of geometry and topology \cites{Blagojevic:2017bl, Barany:2018fy}.

\begin{theorem}[Tverberg 1966]
    Given $(r-1)(d+1)+1$ points in $\rr^d$, there exists a partition of them into $r$ parts whose convex hulls intersect. 
\end{theorem}

Tverberg's theorem is a discrete generalization of Rado's centerpoint theorem, as every point in the intersection of the convex hulls of a Tverberg partition of a set is also a centerpoint of that set.  Just as the Tverberg's theorem generalizes Rado's centerpoint theorem, the Tverberg--Vre\'cica conjecture aims to generalize the central transversal theorem in the same way.  It states the following.

\begin{conjecture}[Tverberg, Vre\'cica 1993]\label{conj:Tverberg-Vrecica}
    Let $0 \le k < d$ be integers, and $r_0, \dots, r_k$ be positive integers.  Assume we are given $k+1$ sets $P_0, \dots, P_k$ of $\rr^d$ such that for each $i$ we have ${|P_i| = (r_i-1)(d-k+1)+1}$.  Then, there exists a $k$-dimensional flat $V\subseteq \rr^d$ and a partition of each $P_i$ into $r_i$ parts such that the convex hull of each part of each partition intersects $V$.
\end{conjecture}

The case $k=0$ is Tverberg's theorem.  Tverberg and Vre\'cica proved the case $k=d-1$ and a slightly modified version of the case $k=d-2$ \cite{Tverberg:1993ia}. The topological generalization of the conjecture, along with a colorful version, has been confirmed by Karasev \cite{Karasev:2007ib} when all $r_i$ are powers of the same prime $p$ and $p(d-k)=p\cdot\codim_\R V$ is even; the conjecture remains open otherwise. See also \cites{Zivaljevic1999, Vrecica2003} for earlier works by Vre\'cica and \v{Z}ivaljevi\'c. Furthermore, optimal colorful version of Conjecture \ref{conj:Tverberg-Vrecica} was proved by Blagojevi\'c, Matschke and Ziegler \cite{blagojevic2011optimal} when $r_0 = \dots = r_k = p$ and $p$ is a prime, and either $p(d-k)$ is even or $k=0$.

The requirement on the parameters by Karasev is not surprising.  Namely, he proved a stronger statement where the case $k=0$ implies the topological version of Tverberg's theorem, which is known to fail if $r_0$ is not a power of a prime \cites{mabillard2014, Frick:2015wp}.

We show that the topological techniques used by Karasev, as well as those by Blagojevi\'c, Matschke and Ziegler, can be extended to prove many cases of a complex analogue of Conjecture \ref{conj:Tverberg-Vrecica} and its optimal colorful version, again when $p \cdot \codim_\R V$ is even. We split them into separate statements depending on parity of $\codim_\R V$.  For a positive integer $N$, let $\Delta_N$ denote an $N$-dimensional simplex with $N+1$ vertices.

\begin{restatable}[Topological Tverberg--Vre\'cica for complex flats]{theorem}{thmTVeven}\label{thm:strong-Tv-vrecica-complex}
    Let $0 \le k < d$ be integers and $p$ be a prime number. For each $i=0, \dots, k$ let $r_i$ be a power of $p$, let $N_i \coloneqq (r_i-1)(2d-2k+1)$ and $f_i \colon \Delta_{N_i} \longrightarrow \C^d$ a continuous map.
    Then, for each $i=0, \dots, k$ there are points $x_1^i, \dots, x^i_{r_i} \in \Delta_{N_i}$ with pairwise disjoint supports and there is an affine complex $k$-dimensional subspace $V \subseteq \C^d$ such that
    \[
        \bigcup_{i=0}^k\{f_i(x_1^i), \dots, f_i(x^i_{r_i})\} \subseteq V.
    \]
\end{restatable}

As a special case when all maps $f_i$ are affine, we obtain the following.

\begin{corollary}[Tverberg--Vre\'cica for complex flats]\label{cor:convex-complex-tv-vrecica}
    Let $0 \le k < d$ be integers and $p$ be a prime number. For each $i=0, \dots, k$ let $r_i$ be a power of $p$ and let $P_i$ be a set of $(r_i-1)(2d-2k+1)+1$ points in $\C^d$.  Then, there exists a complex $k$-dimensional affine space $V$ of $\C^d$ and a partition of each $P_i$ into $r_i$ parts such that the (real) convex hull of each part of each $P_i$ intersects $V$.
\end{corollary}

Even though we require each $r_i$ to be a power of the same prime $p$, Corollary \ref{cor:convex-complex-tv-vrecica} generalizes Theorem \ref{thm:complex-central-transversal} by a standard approximation argument.  This is because any measure $\mu$ can be approximated by discrete measures supported on finite sets of points, and imposing these sets to have a cardinality a power of $p$ is not a problem.  If we denote $N_i = (r_i-1)(2d-2k-1)+1$, then $r_i = \left\lceil N_i/(2d-2k+1)\right\rceil$.  Given a Tverberg--Vre\'cica partition of a set $P_i$, every half-space containing $V$ must have at least one point of each part of the partition, and therefore contain at least a $(1/(2d-2k+1))$-fraction of $P_i$.  Taking $r_i \to \infty$, we can use the set $P_i$ to approximate a given measure $\mu_i$, and by doing this simultaneously for $k+1$ measures, we obtain the conclusion of Theorem \ref{thm:complex-central-transversal}.

Just as with the real case of the Tverberg--Vre\'cica conjecture, determining if the corollary above holds without conditions on the parameters $r_i$ is an interesting problem.

Next, we state the complex analogue of the optimal colorful version of Conjecture \ref{conj:Tverberg-Vrecica}.

\begin{restatable}[Optimal colorful Tverberg--Vre\'cica for a complex flats]{theorem}{thmTVcolored}
\label{thm:colored-tverberg-vrecica}
    Let $0 \le k < d$ be integers and $p$ a prime number. For each $i=0, \dots, k$, let $P_i$ be a set of $(r_i-1)(2d-2k+1)+1$ points in $\C^d$ colored in such a way that the color classes are of cardinality at most $p-1$.  Then, there exists a complex $k$-dimensional affine space $V$ of $\C^d$ and a partition of each $P_i$ into $p$  rainbow parts (in the sense that each part of $P_i$ contains at most one point of each color) such that the (real) convex hull of each part of each $P_i$ intersects $V$.
\end{restatable}

We also prove a topological version of the complex Tverberg-Vre\'cica conjecture when $\codim_\R V$ is odd and $p=2$. Here we can require $V$ to have a complex structure on one of its real co-dimension one subspaces.

\begin{restatable}[Topological Tverberg--Vre\'cica for a complex flat plus a real line]{theorem}{thmTVodd}
\label{thm:strong-tverberg-vrecica-odd-flats}
    Let $0 \le k < d$ be integers. For each $i=0, \dots, k$, let $r_i$ be a power of 2, $N_i \coloneqq (r_i-1)(2d-2k+2)$ and $f_i \colon \Delta_{N_i} \longrightarrow \C^d$ a continuous map.
    Then, there is a real $(2k-1)$-dimensional affine space $V$ of $\C^d$ which is a Minkowski sum of a complex $(k-1)$-dimensional flat and a real one-dimensional flat and, for each $i=0, \dots, k$, there are points $x_1^i, \dots, x^i_{r_i} \in \Delta_{N_i}$ having pairwise disjoint supports, such that
    \[
        \bigcup_{i=0}^k\{f_i(x_1^i), \dots, f_i(x^i_{r_i})\} \subseteq V.
    \]
\end{restatable}

Again, by restricting our attention to the case when all maps $f_i$ are affine, we obtain:

\begin{corollary}[Tverberg--Vre\'cica for complex flats plus a real line]\label{cor:convex-tverberg-vrecica-odd-flats}
    Let $0 \le k < d$ be integers. For each $i=0, \dots, k$ let $r_i$ be a power of two and let $P_i$ be a set of $(r_i-1)(2d-2k+2)+1$ points in $\C^d$.  Then, there exists a real $(2k-1)$-dimensional affine space $V$ of $\C^d$ which is a Minkowski sum of a complex $(k-1)$-dimensional flat and a real one dimensional flat, and a partition of each $P_i$ into $r_i$ parts such that the (real) convex hull of each part of each $P_i$ intersects $V$.
\end{corollary}

This version also implies a complex central transversal theorem via the same standard approximation argument.  Note that the theorem below is neither a consequence nor a generalization of Theorem \ref{thm:complex-central-transversal}, which provides a stronger guarantee on the depth of the transversal, yet it affects fewer half-spaces than the theorem below.

\begin{restatable}{theorem}{thmCentralodd}
\label{thm:complex-central-transversal-odd}
    Let $1 \le k < d$ be integers and let $\mu_0, \dots, \mu_{k}$ be probability measures in $\C^d$, which are absolutely continuous with respect to the Lebesgue measure.  Then, there exists a real $(2k-1)$-dimensional affine space $V$ of $\C^d$ which is a Minkowski sum of a complex $(k-1)$-dimensional flat and a real one dimensional flat such that for every closed affine real space $H \subseteq \C^d$ containing $V$ we have 
    \[
    \mu_i (H) \ge \frac{1}{2d-2k+2}
    \]
    for each $i = 0, \dots, k$.
\end{restatable}

The rest of the manuscript is structured as follows. In Section \ref{sec:notation}, we establish the notation and basic facts used in our proofs. In Sections \ref{sec:central transversal} and \ref{sec:fairy bread}, respectively, we prove the complex analogues of the central transversal theorem and the fairy bread sandwich theorem by taking Theorem \ref{thm:equiv} as a black box, which is proven in Section \ref{sec:equivariant}.  We also give a direct proof of Theorem \ref{thm:complex-central-transversal-odd} in Section \ref{sec:central transversal} using the configuration space -- test map scheme, without relying on modifications of the Tverberg--Vre\'cica conjecture.  Then, in Section \ref{sec:tverberg-vrecica} we prove several complex versions of the Tverberg--Vre\'cica are a consequence the non-vanishing the Euler classes of certain vector bundles.  The non-vanishing of said classes is proved in Section \ref{sec:euler-classes}.

\section{Notation and conventions}\label{sec:notation}

In this section, we establish the notation and key facts that will be used throughout the paper.

\subsection*{Real and complex structure} We consider $\C^d$ as a complex vector space equipped with the complex inner product $\langle u, v \rangle_\C = \sum_{j=1}^d u_j \overline{v}_j$. Moreover, for any $v \in \C^d$, we have
\[
	\spanC\{v\} = \spanR\{v, iv\} \textrm{ and } \{v\}^{\perp_{\C}} = \{v, iv\}^{\perp_{\R}}.
\]
In particular, it follows that for vectors $v_1, \dots, v_k \in \C^d$ the complex-orthogonal direct sum
\[
		\C^d = \spanC\{v_1, \dots, v_{k}\} \oplus_{\C} \{v_1, \dots, v_{k}\} ^{\perp_{\C}}
\]
is also a real-orthogonal direct sum
\[
	\R^{2d} = \spanR\{v_1, iv_1, \dots, v_{k}, iv_{k}\} \oplus_{\R} \{v_1, iv_1, \dots, v_{k}, iv_{k}\}^{\perp_{\R}}.
\]

\subsection*{Equivariant cohomology and Fadell--Husseini index}

For a finite group $G$ and a $G$-space $X$, we denote by $X\times_G EG \coloneqq(X \times EG)/G$ the Borel construction. Here $EG$ denotes a contractible free $G$-space. 
Equivariant cohomology with coefficients in a ring $R$ is defined as 
\[
    H^*_G(X;R) \coloneqq H^*(EG\times_G X;R).
\]
In the special case when $X$ is a point, we have $H^*_G(\pt;R) = H^*(BG;R)$, where $BG \coloneqq EG/G$ is the classifying space. From functoriality of cohomology with respect to continuous maps, it follows that equivariant cohomology is functorial with respect to $G$-maps: a $G$-map $X \longrightarrow Y$ induces a ring morphism $H^*_G(Y;R) \longrightarrow H^*_G(X;R)$.
The first projection $X \times EG \longrightarrow X$ induces a map $X\times_G EG \longrightarrow X/G$. If the action of $G$ on a CW-complex $X$ is free, the latter map is a homotopy equivalence.
The second projection $X \times EG \longrightarrow EG$ induces a Borel fibration
\[
    X \longrightarrow X \times_G EG \xrightarrow{~~\pi~~} BG,
\]
and the Fadell--Husseini index of $X$ is defined as
\[
    \FHindex{G}{R}{X}  \coloneqq \ker \Big(\pi^* \colon H^*(BG;R) \longrightarrow H^*(EG\times_G X;R) \Big) = \ker \Big(H^*_G(\pt;R) \longrightarrow H^*_G(X;R) \Big),
\]
where the second map is induced by the $G$-map $X \longrightarrow \pt$. By functoriality of equivariant cohomology, the following implication holds:
\begin{equation*}
    \exists~\text{a}~G\text{-map}~ X \longrightarrow Y ~ \Longrightarrow~ \FHindex{G}{R}{Y} \subseteq \FHindex{G}{R}{X},
\end{equation*}
also known as the {\em monotonicity} of the Fadell--Husseini index.

\subsection*{Grassmannians}
For integers $0 \le k \le d$, we will denote by $\grassmannR{k}{d}$ and $\grassmannC{k}{d}$ the real, resp. complex, Grassmannian of $k$-dimensional subspaces in a $d$-dimensional space. Moreover, we will denote by $\gamma_k^\R$ and $\gamma_k^\C$ the corresponding canonical real, resp. complex, vector bundles and by $(\gamma_k^\R)^{\perp}$ and $(\gamma_k^\C)^{\perp}$ their orthogonal complements. Cohomology of the real Grassmannian with $\F_2$-coefficients \cite{borel1953cohomologieMod2} is given by
\begin{equation*}
    H^*(\grassmannR{k}{d};\F_2) \cong \F_2[w_1, \dots, w_k]/I_k,
\end{equation*}
where $|w_i|=i$, for each $i=1, \dots, k$, and $I_k$ is an ideal on $k$ generators, which are the homogeneous parts of the power series 
\[
    (1+w_1+\dots+w_k)^{-1} = 1+ (w_1+\dots+w_k) + (w_1+\dots+w_k)^2 + \dots \in \F_2[[w_1, \dots, w_k]]
\]
of degrees $d-k+1, \dots, d$. Moreover, the generators $w_i$ are the Stiefel--Whitney classes of the canonical bundle $\gamma_k^\R$.
Analogously, integral cohomology of the complex Grassmannian \cite{borel1953cohomologieFibres} is given (with a slight abuse of notation) as
\begin{equation*}
    H^*(\grassmannC{k}{d};\Z) \cong \Z[c_1, \dots, c_k]/I_k,
\end{equation*}
where $|c_i|=2i$, for each $i=1, \dots, k$, and $I_k$ is an ideal on $k$ generators, which are the homogeneous parts of the power series
\[
    (1+c_1+\dots+c_k)^{-1} = 1 - (c_1+\dots+c_k) + (c_1+\dots+c_k)^2 - \dots \in \Z[[c_1, \dots, c_k]].
\]
of degrees $2(d-k+1), \dots, 2d$. Additionally, the generators $c_i$ are the Chern classes of the canonical bundle $\gamma_k^\C$.

Let $W_k(\C^d)$ denote the Stiefel manifold of orthonormal $k$-frames and let $U(k)$ be the unitary group, which we assume acts by matrix multiplication on $W_k(\C^d)$. Then, the canonical bundle $\gamma_k^\C$ can be seen as the pullback
\begin{equation}\label{eq:tautol-bundle-pullback}
    \begin{tikzcd}
        W_k(\C^d)\times_{U(k)} \C^k \arrow[r, hook] \arrow[d] \arrow[dr, phantom, "\lrcorner", very near start] & W_k(\C^{\infty})\times_{U(k)} \C^k \arrow[d] \\
        \grassmannC{k}{d}=W_k(\C^d)/U(k) \arrow[r, hook] & W_k(\C^{\infty})/U(k)=\grassmannC{k}{\infty}
    \end{tikzcd}
\end{equation}
along the map induced by the inclusion $W_k(\C^d) \hookrightarrow W_k(\C^{\infty})$. Since $W_k(\C^{\infty})$ is contractible with a free $U(k)$-action, it is a model for $EU(k)$ and $\grassmannC{k}{\infty}=BU(k)$. Recall that 
\[
    H^*(BU(k);\Z) \cong \Z[c_1, \dots, c_k],
\]
where $c_i$ is of degree $2i$ and is the $i$-th Chern class of the tautological bundle (i.e., the right vertical map in \eqref{eq:tautol-bundle-pullback}). Moreover, the bottom horizontal map in \eqref{eq:tautol-bundle-pullback} is the classification map of $\gamma_k^\C$ and induces the canonical projection
\[
    H^*(\grassmannC{k}{\infty}; \Z) \cong \Z[c_1, \dots, c_k] \relbar\joinrel\twoheadrightarrow \Z[c_1, \dots, c_k]/I_k \cong H^*(\grassmannC{k}{d}; \Z).
\]
Next, diagonal inclusion $U(1)^k \subseteq U(k)$ induces the pullback diagram of bundles
\begin{equation} \label{eq:map-rk}
    \begin{tikzcd}
        W_k(\C^{\infty})\times_{U(1)^k} \C^k \arrow[r] \arrow[d] \arrow[dr, phantom, "\lrcorner", very near start] & W_k(\C^{\infty})\times_{U(k)} \C^k \arrow[d] \\
        BU(1)^k = W_k(\C^{\infty})/U(1)^k \arrow[r, "r_k"] & W_k(\C^{\infty})/U(k) = BU(k)
    \end{tikzcd}
\end{equation}
along the quotient map denoted by $r_k$. The bundle on the right hand side of \eqref{eq:map-rk} is again the canonical bundle with the total Chern class $1+c_1+\dots + c_k \in H^*(BU(k); \Z) \cong \Z[c_1, \dots, c_k]$.

On the other hand, the bundle on the left hand side of \eqref{eq:map-rk} splits as a direct sum of $k$ line bundles $\xi_1 \oplus \dots \oplus \xi_k$, which are defined below in \eqref{eq:xi_i-bundle}. This is so because $\C^k$, as a $U(1)^k$-representation, splits into $k$ one-dimensional complex subrepresentations with the product action. Namely, the canonical complex line bundle
\begin{equation} \label{eq:canonical-line}
    \C \longrightarrow EU(1) \times_{U(1)} \C \longrightarrow BU(1)
\end{equation}
has the total Chern class $1+u \in H^*(BU(1);\Z) \cong \Z[u]$, where $|u|=2$. Consequently, if, for each $i=1, \dots, k$, we endow $\C$ with an action of the $i$-th factor of $U(1)^k$ by complex multiplication (while the other factors acts trivially), the bundle
\begin{equation} \label{eq:xi_i-bundle}
    \xi_i \colon~~\C \longrightarrow EU(1)^k \times_{U(1)^k} \C \longrightarrow BU(1)^k
\end{equation}
is the pullback of \eqref{eq:canonical-line} along the projection on the $i$-th factor $BU(1)^k \longrightarrow BU(1)$. Therefore, the total Chern class of $\xi$ equals
\[
    c(\xi_i) = 1+u_i \in H^*(BU(1)^k;\Z)\cong \Z[u_1, \dots, u_n],
\]
where $|u_1| = \dots =|u_k| = 2$. Finally, since the left-hand bundle in \eqref{eq:map-rk} is isomorphic to $\xi_1 \oplus \dots \oplus \xi_k$, by the Whitney sum formula, its total Chern class is $(1+u_1)\dots (1+u_k) \in \Z[u_1, \dots, u_k]$. Therefore, by naturality of Chern classes,
\[
    r_k^* \colon \Z[c_1, \dots, c_k] \longrightarrow \Z[u_1, \dots, u_k]
\]
maps $c_j$ to the degree $2j$ homogeneous part of $(1+u_1)\dots (1+u_k)$. In particular, we have $r_k^*(c_k) = u_1\dots u_k$.

\section{Complex central transversals}
\label{sec:central transversal}

In this section, we prove Theorem \ref{thm:complex-central-transversal}. Let us recall the statement. 

\thmCentral*

In the proof we use Theorem \ref{thm:equiv}, whose proof is deferred to Section \ref{sec:equivariant}. In fact, we obtain a bit more general Theorem \ref{thm:central-for-flags}, which is a complex analogue of \cite[Thm.~5.2]{Manta2024}.\\

Let us introduce a bit of notation.
For a measure $\mu$ in $\C^d$, we will say that a complex affine $k$-dimensional subspace $V \subseteq \C^d$ is its \textit{complex central $k$-transversal} if $\mu(H) \ge 1/(2d-2k+1)$ for every real-halfspace $H \supseteq V$ in $\C^d$. For $k=0$, we will say a {\it centerpoint} of $\mu$ instead of a complex central $0$-transversal. The set of centerpoints of a given measure is convex. Using this language, the complex central transversal theorem says that any $k+1$ measures in $\C^d$ have a common complex central $k$-transversal.\\

We now state and prove an extension of Theorem \ref{thm:complex-central-transversal}.

\begin{theorem} \label{thm:central-for-flags}
    Let $0 \le k < d$ be integers and $\mu_0, \dots, \mu_{d-1}$ probability measures in $\C^d$, which are absolutely continuous with respect to the Lebesgue measure. Then, there exists a flag of complex affine spaces
    \[
    V_k \subseteq V_{k+1} \subseteq \dots \subseteq V_{d-1} \subseteq \C^d
    \]
    of respective dimensions $k, k+1, \dots, d-1$, such that $V_k$ is a complex central $k$-transversal for $\mu_0, \dots, \mu_k$ and $V_i$ is a complex central $i$-transversal for $\mu_i$, for each $i = k+1, \dots, d-1$. 
\end{theorem}
\begin{proof}
	Let $W_{d-k}(\C^d)$ denote the complex Stiefel manifold of complex orthonormal frames and let $(v_1, \dots, v_{d-k}) \in W_{d-k}(\C^d)$. For each $j = 1, \dots, d-k$ we define $U_j \coloneqq \spanC \{v_1, \dots, v_j\}$ and let
	\[
		\pi^j \colon \C^d \longrightarrow U_j,~z \longmapsto \langle z, v_1 \rangle v_1 + \dots +  \langle z, v_j \rangle v_j
	\]
	be the complex-orthogonal projection.
	
	For each $i=0, \dots, k$ we will denote by $p_i$ the barycenter of the set of centerpoints of the push-forward measure $\pi^{d-k}_*\mu_i$ in $U_{d-k}$. That is, for any real halfspace $T \subseteq U_{d-k}$ that contains $p_{i}$ we have 
	\[
		(\pi^{d-k}_*\mu_i)(T) \ge \frac{1}{\dim_{\R}(U_{d-k})+1} = \frac{1}{2d-2k+1}.
	\]
    Similarly to that, for each $i=k+1, \dots, d-1$, let $p_i \in U_{d-i}$ be the barycenter of all centerpoints of  $\pi^{d-i}_*\mu_i$ in $U_{d-k}$, so
    \[
		(\pi^{d-i}_*\mu_i)(T) \ge \frac{1}{\dim_{\R}(U_{d-i})+1} = \frac{1}{2d-2i+1},
	\]
    for every real halfspace $T$ in $U_{d-i}$ containing $p_i$.
    
	We define a test map
	\begin{align*}
        f \colon~ W_{d-k}(\C^d) &\longrightarrow \C^{d-1} \oplus \dots \oplus \C^{k},\\
        (v_1, \dots, v_{d-k})& \longmapsto \big(\langle v_1, p_{0}-p_{i}\rangle\big)_{i=1}^{d-1} \oplus \dots \oplus \big(\langle v_{d-k}, p_{0}-p_{i}\rangle\big)_{i=1}^{k}.
    \end{align*}
    By Theorem \ref{thm:equiv} we conclude that $f$ hits the origin, i.e.,
	there exists $(v_1, \dots, v_{d-k}) \in W_{d-k}(\C^d)$ for which $p_0 = \dots = p_k$ and, for each $i=k+1, \dots, d-1$, we have $p_i = \pi^{d-i}(p_0)$. 
	Finally, $V_i \coloneqq p_{0} + (U_{d-i})^{\perp} = p_{i} + (U_{d-i})^{\perp}$, for $i=k, \dots, d-1$, is the desired affine complex flag.
\end{proof}

\begin{remark} [Optimality of complex central transversal theorem]
\label{rem:optimality}   
	Once the complex dimension $k$ of an affine subspace is fixed, Theorem \ref{thm:complex-central-transversal} is optimal with respect to both the number of measures considered and the fraction of each measures contained in halfspaces. Let $e_1, \dots, e_d$ denote the standard complex basis of $\C^d$.
	
	Indeed, to see why a complex $k$-dimensional central transversal for $k+2$ measures need not exist, consider the following example: take uniform measures on balls of radius $\varepsilon > 0$, with centers at $k+2$ complex-affinely independent points $0, e_1, \dots, e_{k+1} \in \C^d$.
	Then, for $\varepsilon$ small enough, any $k$-flat would contain $k+2$ complex-affinely independent points, which leads to a contradiction. 
	
	On the other hand, to see that the fraction cannot be improved, consider $\mu_1, \dots ,\mu_k$ which are uniform measures on balls of radius $\varepsilon > 0$, having as centers $e_1, \dots, e_k$.
	As for the measure $\mu_{0}$, assume it is a uniform measure on $2d-2k+1$ disjoint balls of radius $\varepsilon^3$ centered at real-affinely independent points
    \[
        x_1 \coloneqq \varepsilon^2 e_{k+1},~  x_2 \coloneqq \varepsilon^2 ie_{k+1},~ \dots, x_{2(d-k)-1} \coloneqq \varepsilon^2 e_{d},~ x_{2(d-k)} \coloneqq \varepsilon^2 ie_{d},~ x_0 \coloneqq - x_1-\dots-x_{2(d-k)},
    \]
    where $i$ denotes the unit imaginary complex number.
    We are going to prove that there exists $\varepsilon >0$ small enough, such that for any complex $k$-transversal $V$ to $\mu_0, \dots, \mu_k$ there exists a real halfspace $H \supseteq V$ for which $\mu_0(H) \le 1/(2d-2k+1)$. To this end, let $V$ be a complex $k$-transversal to $\mu_0, \dots, \mu_k$. Then, it has to intersect the supports of $\mu_1, \dots, \mu_k$; otherwise, if $V \cap \supp \mu_i = \emptyset$, some real halfspace containing $V$ would have measure zero with respect to $\mu_i$. Similarly, $V$ also intersects the convex hull of $\supp \mu_0$. For each $j=1,\dots,k$ we denote by $z_j$ the orthogonal projection of $e_j$ to $V$ and by $z_0$ the orthogonal projection of the origin to $V$ (these projections are the closest points of $V$ to $e_j$'s and the origin, respectively). Then, since $V$ intersects the convex hulls of supports of $\mu_j$'s we have $\|e_j-z_j\| \le \varepsilon$, for all $j=1, \dots, k$, and $\|z_0\| \le \varepsilon$. Let $U \coloneqq \linspan_{\C}\{e_{k+1}, \dots, e_d\} \subseteq \C^d$. We have the following lemma.
    \begin{lemma}
        Let $0 \le k < d$ be integers and $\varepsilon >0$ a real parameter. We denote by $B_{\varepsilon}(x)$ the closed $\varepsilon$-ball around a point $x \in \C^d$. Then, we have the following.
        \begin{enumerate}[(i)]
            \item If $\varepsilon>0$ small enough, any tuple $y=(y_0,y_1, \dots, y_k)\in B_{\varepsilon}(0) \times B_{\varepsilon}(e_1) \times \dots \times B_{\varepsilon}(e_k)$ consists of $k+1$ complex affinely independent points in $\C^d$ and their complex affine span $V_y$ is transversal to $U$ with a non-empty intersection. In particular, we have $\{q_y\} = V_y \cap U$.
            \item If $\varepsilon>0$ small enough, for any $y\in B_{\varepsilon}(0) \times B_{\varepsilon}(e_1) \times \dots \times B_{\varepsilon}(e_k)$ the unique linear projection $\pi_y \colon~ \C^d \to U$ such that $\pi_y^{-1}(q_y) = V$ maps the unit ball $B_1(0) \subseteq \C^d$ inside the ball of radius 2 around the origin in $U$.
        \end{enumerate}
    \end{lemma}
    \begin{proof}
        Points $y_0, \dots, y_k \in \C^d$ are affinely independent if and only if $(y_0, 1), \dots, (y_k,1) \in \C^{d+1}$ are linearly independent. The latter is equivalent to the fact that some $k$-minor of the matrix $((y_0, 1), \dots, (y_k,1))$ is non-zero. Thus, being affinely independent is an open condition (i.e., presereved under small perturbations). Similarly, due to complementary dimensions, $V_y$ being transversal to $U$ with a non-empty intersection is also an open condition, as it means that the vectors $y_1-y_0, \dots, y_k-y_0,e_{k+1},\dots, e_d$ are linearly independent. Therefore part (i) follows since $0, e_1, \dots, e_k$ are affinely independent and $V_{(0,e_1, \dots, e_k)} = \linspan_{\C}\{e_1, \dots, e_k\}$ is transversal to $U$ with the intersection at the origin.

        As for part (ii), let us denote by $\mathcal{L}(\C^d, U)$ the space of linear maps $L \colon~ \C^d \to U$ endowed with the operator norm $\|L \| = \inf\{c \ge 0 \colon~ \|L(v)\| \le c\|v\|~\text{for all } v\in\C^d\}$. Recall that any $L \in \mathcal{L}(\C^d, U)$ has a finite norm. Then, from the definition of the norm, we observe that $L(B_1(0))$ is contained in the $\|L\|$-ball around the origin in $U$. Therefore, we need to prove there exists $\varepsilon$ small enough, such that for any $y \in B_{\varepsilon}(0) \times B_{\varepsilon}(e_1) \times \dots \times B_{\varepsilon}(e_k)$ we have $\|\pi_y\| \le 2$.\\
        By part (i), there is $\varepsilon >0$ small enough such that we have a well-defined and continuous map
        \begin{equation*}
            \Pi \colon~B_{\varepsilon}(0) \times B_{\varepsilon}(e_1) \times \dots \times B_{\varepsilon}(e_k) \longrightarrow \mathcal{L}(\C^d,U),~y \longrightarrow \pi_{y}.
        \end{equation*}
        Since $\Pi(0, e_1, \dots, e_k)$ is the orthogonal projection, we have $\|\Pi(0, e_1, \dots, e_k)\|=1$. Let $B$ denote the open ball of radius 1 in $\mathcal{L}(\C^d,U)$ around $\Pi(0, e_1, \dots, e_k)$. From the triangle inequality, we have $\|L\| < 2$, for any $L \in B$. Moreover, $\Pi^{-1}(B) \ni (0, e_1, \dots, e_k)$ is an open set, hence there exists $\varepsilon' < \varepsilon$ such that
        \[
            B_{\varepsilon'}(0) \times B_{\varepsilon'}(e_1) \times \dots \times B_{\varepsilon'}(e_k) \subseteq \Pi^{-1}(B).
        \]
        In other words, for any $y \in B_{\varepsilon'}(0) \times B_{\varepsilon'}(e_1) \times \dots \times B_{\varepsilon'}(e_k)$, we have $\pi_y = \Pi(y) \in B$ and, consequently, $\|\pi_y\| < 2$, as observed observed above. This completes the proof of the lemma.
    \end{proof}
    In particular, by part (i) of the lemma, we have that, for $z=(z_0,z_1,\dots,z_k)$ and for $\varepsilon>0$ small enough, the corresponding $k$-plane $V_z=V$ is transversal to $U$ with $\{q_z\} = V \cap U$. Moreover, by part (ii), for $\varepsilon>0$ small enough, $\pi_{z} \colon~ \C^d \to U$ maps the unit ball in $\C^d$ to the ball of radius 2 in $U$. In particular, since $x_j \in U$, $\pi_{z}$ maps the $\varepsilon^3$-balls in $\C^d$ around $x_j$ (i.e., the balls that constitute the support of $\mu_0$) to $2\varepsilon^3$-balls in $U$ around $x_j$. Therefore, the support of the pushforward measure $(\pi_z)_*(\mu_0)$ in $U$ is contained in the union of $2\varepsilon^3$-balls centered around points $x_0, \dots, x_{2(d-k)} \in U$.
    \begin{figure}[h]
        \centering
        \subfloat[\centering]{{\includegraphics[width=5cm]{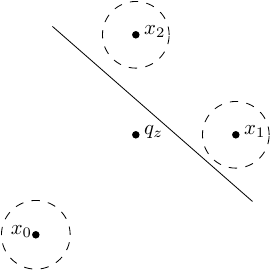} }}%
        \qquad
        \subfloat[\centering]{{\includegraphics[width=5cm]{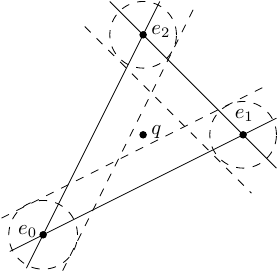} }}%
        \caption{}%
        \label{fig:optimality}%
    \end{figure}
    Now, it is enough to show that there exists a real affine half-space $T$ in $U$ and some $j=0, \dots, 2(d-k)$ for which $x_{j}$ and $q_z$ are on one side, while the $2\varepsilon^3$-balls around the points $x_l$, for all $l \neq j$, are on the other side. See Figure \ref{fig:optimality}(A) for illustration. Indeed, then we would have $(\pi_z)_*(\mu_0)(T) \le 1/(2d-2k+1)$, so the real half-space $\pi_{z}^{-1}(T) \subseteq \C^d$ would satisfy $\mu_0(\pi_{z}^{-1}(T)) \le 1/(2d-2k+1)$, as required. Finally, to prove the said separation claim in $U$, we first notice it is scale invariant, so we may scale it (via homothety centered at the origin of $U$) by a factor of $1/\varepsilon^2$. At this point, we can apply directly the lemma below, thus completing the remark about the optimality of the fraction in Theorem \ref{thm:complex-central-transversal}.

    \begin{lemma}
        Let $m \ge 2$ be an integer, $e_1, \dots, e_m \in \R^m$ the standard real basis, $e_0 \coloneqq - e_1- \dots -e_m \in \R^m$, and let $R=R(m)$ be the radius of the inscribed sphere of the $m$-simplex $\conv\{e_0, \dots, e_m\}$. Then, for any $0 < r < R $ and any $q \in \R^m$, there exists a real half-space $H \subseteq \R^m$ and an integer $j=0,\dots,m$, such that $e_j,q \in H$ , while the open $r$-balls around points $e_l$, for each $l \neq j$, are disjoint from $H$.
    \end{lemma}
    \begin{proof}
        For each $j=0, \dots, m$ let $H_j$ be the real half-space containing $e_j$ in the interior whose boundary $\partial H_j$ is the affine hyperplane affinely spanned by $e_0, \dots, e_{j-1}, e_{j+1}, \dots, e_m$. From affine independence of points it follows that $H_j$'s are well defined and $H_0 \cup \dots \cup H_m = \R^m$. For each $j=0,\dots, m$, we denote by $H_{j,r}$ the translated half-space $H_j$ by the vector of norm $r$ orthogonal to $\partial H_j$ in the positive direction. Then, since $r<R$, we also have $H_{0,r} \cup \dots \cup H_{m,r} = \R^m$, so there exists some $j$ such that $q \in H_{j,r}$. See Figure \ref{fig:optimality}(B) for illustration. Finally, $H_{j,r}$ is the required half-space: $q, e_j \in H_{j,r}$ and the open $r$-balls around $e_l$, for $l\neq j$, are disjoint from $H_{j,r}$.
    \end{proof}
\end{remark}

We also give a proof sketch of the odd codimension version of complex central transversal theorem, Theorem \ref{thm:complex-central-transversal-odd}, using the configuration space -- test map scheme. We first recall the statement.

\thmCentralodd*
\begin{proof}
    The proof idea is very similar to that of Theorem \ref{thm:complex-central-transversal}. For $(v_0, \dots, v_{d-k}) \in W_{d-k+1}(\C^d)$ let $U \coloneqq \linspan_\R\{v_0\} \oplus \linspan_\C\{v_1, \dots, v_{d-k}\} $ and denote by
    \[
		\pi \colon \C^d \longrightarrow U,~z \longmapsto \langle z, v_0 \rangle_\R v_0 + \langle z, v_1 \rangle_\C v_1 + \dots +  \langle z, v_{d-k} \rangle_\C v_{d-k} 
	\]
	the projection. For each $0 \le i \le k$ let $p_{i}$ be the centerpoint of $\pi_*\mu_i$ in $U$, that is,
    \[
        \pi_*\mu_i(H) \ge \frac{1}{2d-2k+2}
    \]
    for every (real) halfspace $H \ni p_{i}$ in $U$.
    The center point is not necessarily unique, but they all form a convex set for a given measure, so we may take $p_{i}$ to be the barycenter.

    We define a test map
	\begin{align*}
		f \colon W_{d-k+1}(\C^d) &\longrightarrow \R^k \oplus (\C^{k})^{ d-k},\\ 
        (v_0, \dots, v_{d-k}) &\longmapsto \big(\langle v_0, p_{j} - p_{0} \rangle_\R\big)_{j=1}^k \oplus \Big(\big(\langle v_i, p_{j} - p_{0} \rangle_\C\big)_{j=1}^k \Big)_{i=1}^{d-k}
	\end{align*}
    which is equivariant with respect to $\Z_2^{d-k+1}$-diagonal action. By Remark \ref{rem:Z_2^n}, it must hit the origin and so $p_{0} = \dots = p_{k}$. Similarly as above, $V \coloneqq p_{i} + U^{\perp_\R}$ is the desired flat.
\end{proof}

The same methods used in Remark \ref{rem:optimality} do not imply the optimality of the depth obtained in Theorem \ref{thm:complex-central-transversal-odd}.  It would be interesting to know if this is indeed optimal.

\section{Complex fairy bread sandwich theorem}
\label{sec:fairy bread}

In this section, we establish a complex analogue of the fairy bread sandwich theorem of Axelrod-Freed and Sober\'on \cite{AxelrodFreed2022} via a configuration space -- test map scheme. The required equivariant result, Theorem \ref{thm:equiv}, is presented in the next section. 

Let $0 < n \le d$ be integers and $A_n(\C^d$) be the set of $n$-dimensional complex affine subspaces of $\C^d$. For $V \in A_n(\C^d)$, let $\M_n(V)$ be the space of finite measures on $V$ which are absolutely continuous with respect to the Lebesgue measure in $V$, equipped with the weak topology. Let 
\[
    E = \{(V, \mu[V]) \colon V \in A_n(\C^d), \mu[V] \in \M_n(V)\}
\]
be the total space of a fiber bundle with the projection map $E \longrightarrow A_n(\C^d),~(V, \mu[V]) \longmapsto V$.
An $n$-dimensional \emph{mass assignment} $\mu$ is a section of this fiber bundle. Given an $n$-dimensional mass assignment $\mu$ and an $n$-dimensional affine subspace $V$ of $\C^d$, we denote by $\mu[V]$ the measure that $\mu$ induces on $V$. We will moreover call $\mu$ nice if each $\mu[V]$ is absolutely continuous with respect to the Lebesgue measure on $V$.\\

We recall the statement of the theorem.

\thmFairy*
\begin{proof}
    Let $(v_1, \dots, v_{d-k+1}) \in W_{d-k+1}(\C^d)$. For each $i= k-1, \dots, d$, let 
    \[
        V_i \coloneqq \linspan_{\C}\{v_1, \dots, v_{d-i}\}^{\perp_{\C}}
    \]
    be a linear subspace of $\C^d$ of dimension $i$. We will define $S_i$ as an affine subspace of $V_{i+1}$ parallel to $V_i$ by induction.

    Set $S_{d} \coloneqq \C^d$. Suppose that $k-1 \le i \le d-1$ and that $S_{i+1}$ has already been defined. We can write $S_{i+1} = s+V_{i+1}$, for a unique vector $s \perp V_{i+1}$. If we assume that $s$ is the origin of $S_{i+1}$, we may speak of orthogonality in $S_{i+1}$. Define a complex affine line 
    \[
        L \coloneqq s+\linspan_{\C}\{v_{d-i}\} \subseteq S_{i+1}
    \]
    and let $\pi \colon S_{i+1} \longrightarrow L$ denote the orthogonal projection.
    For each of the pushout measures $\pi_*\mu_0^i[S_{i+1}], \dots, \pi_*\mu_{\pi_i}^i[S_{i+1}]$ in $L$, let $p^i_{0}, \dots, p^i_{\pi_i} \in L$ denote the corresponding (real) center point. In other words, for each $j=0, \dots, \pi_i$ and any (real) halfspace $N$ in $L$ containing $p^i_j$, we have 
    \[
        (\pi_*\mu_j^i[S_{i+1}])(N) \ge 1/3.
    \]
    (If for some measure there are multiple center points, we can choose the barycenter, since they form a convex set.) 
    
    We set $S_i$ to be the orthogonal complement of $L$ inside $S_{i+1}$ going through $p^i_{0}$. Therefore, $S_{i}$ is a central transversal to the measure $\mu_{0}^i[S_{i+1}]$: for any (real) halfspace $H$ in $S_{i+1}$ containing $S_i$, we have
    \[
        \mu^i_{\pi_i+1}[S_{i+1}](H) = (\pi_*\mu^i_{\pi_i+1}[S_{i+1}])(H \cap L) \ge \frac{1}{3}.
    \]
    If $p^i_{0}= \dots =p^i_{\pi_i}$, then $S_i$ is also a central transversal for $\mu_1^i[S_{i+1}], \dots, \mu_{\pi_i}^i[S_{i+1}]$ in $S_{i+1}$. However, since the points are in $L$, this condition is equivalent to $\langle v_{d-i}, p_0^i \rangle = \dots = \langle v_{d-i}, p_{\pi_i}^i \rangle$.

    Now, having done this for every $i = k-1, \dots, d-1$, we may define a test map
    \begin{align*}
        f \colon~ W_{d-k+1}(\C^d) &\longrightarrow \bigoplus_{i=k-1}^{d-1} \C^{\oplus \pi_i},\\
        (v_1, \dots, v_{d-k+1}) &\longmapsto \Big(\langle v_i, p_1^i - p_{0}^i\rangle, \dots , \langle v_i, p_{\pi_i}^i - p_{0}^i \rangle \Big)_{i=k-1}^{d-1}.
    \end{align*}
    We endow the domain and the codomain by the product $(S^1)^{d-k+1}$-action. Thus, $f$ is an $(S^1)^{d-k+1}$-equivariant map.

    By Theorem \ref{thm:equiv}, we know that $f$ hits the origin. This means that for a zero $(v_1, \dots, v_{d-k+1}) \in W_{d-k+1}(\C^d)$ of $f$ we have $p_0^i = \dots = p_{\pi_i}^i$ for each $i=k-1,\dots, d-1$. As explained above, the corresponding affine subspaces $S_{d-1} \supseteq \dots \supseteq S_{k-1}$ are the desired ones.
\end{proof}

\section{Equivariant maps and complex Stiefel manifolds}
\label{sec:equivariant}

In this section we prove that every $(S^1)^n$-equivariant map
\[
    W_n(\C^d) \longrightarrow \C^{d-1} \oplus \dots \oplus \C^{d-n}
\]
hits the origin, assuming the product group action. We use the Fadell--Husseini index \cite{Fadell:1988tm} with respect to the full group $(S^1)^n$.
We also include the treatment of the Fadell--Husseini index computations with respect to the subgroup $\Z_2^n \subseteq (S^1)^n$ in Remark \ref{rem:Z_2^n}, thus providing a slightly stronger equivariant statement.\\

Let us first introduce some notation. For integers $n \ge 1$ and for each $i = 1, \dots, n$, we will denote by $\C_i$ the $(S^1)^n$-representation
\[
    (S^1)^n \times \C \longrightarrow \C,~(\lambda_1, \dots, \lambda_n, z) \longmapsto \lambda_i z.
\]
If $1\le i<j \le n$, we denote by $\complexModule{i,j}$ the $(S^1)^n$-representation $\complexModule{i} \otimes_\C \overline{\complexModule{j}}$, where the overline represents action by complex conjugation.
For an integer $d \ge 1$, we consider the complex Stiefel manifold $W_n(\C^d)$ with an $(S^1)^n$-action from 
\[
    W_n(\C^d) \subseteq S(\complexModule{1}^d) \times \dots \times S(\complexModule{n}^d),
\]
where each $S(\complexModule{i}^d)$ denotes the unit sphere in $\complexModule{i}^d$. The bundle $\xi_i$ from \eqref{eq:xi_i-bundle} can be written as
\begin{equation} \label{eq:xi_i bundle}
    \xi_i \colon~~\complexModule{i} \longrightarrow \complexModule{i} \times_{(S^1)^n} E(S^1)^n \longrightarrow B(S^1)^n,
\end{equation}
having the total Chern class $c(\xi_i) = 1+u_i \in \Z[u_1, \dots, u_n]$. We will denote by $\overline{\xi}_i$ the complex conjugate of $\xi_i$. Then, $\xi_i \otimes \overline{\xi}_i$ is the trivial complex line bundle, so its total Chern class equals $c(\overline{\xi}_i) = 1-u_i \in \Z[u_1, \dots, u_n]$ \cite[Prop.~3.10]{hatcher2003vector}.\\

We may now recall the statement of Theorem \ref{thm:equiv} before proving it.

\thmEquiv*
\begin{proof}
    We will use an idea similar to the one in \cite{ChanChenFrickHull20}. Namely,
    for any such equivariant map $f$, we have a $(S^1)^n$-diagram
    \begin{equation*}
        \begin{tikzcd}
            W_n(\C^d) \arrow[rr, "f"] \arrow[d, hook] & & \bigoplus_{i=1}^n \complexModule{i}^{d-i} \arrow[d, hook]\\
            \prod_{i=1}^nS(\complexModule{i}^d) \arrow[rr, "g"] \arrow[rru, dashed] && \bigoplus_{i=1}^n \complexModule{i}^{d-i} \oplus \bigoplus_{1 \le i < j \le n}\complexModule{i,j},
        \end{tikzcd}
    \end{equation*}
    where the dashed map is any equivariant extension of $f$ (which is exists by equivariant Tietze's extension theorem of Gleason \cite{gleason1950spaces}) and $g$ is a direct sum of this extension and the $(S^1)^n$-map
    \[
        \prod_{i=1}^nS(\complexModule{i}^d) \longrightarrow \bigoplus_{1 \le i < j \le n}\complexModule{i,j},~ (z_i)_{i=1}^n \longmapsto (\langle z_i, z_j \rangle)_{1 \le i < j \le n}.
    \]
    Since $f$ and $g$ have identical zero sets, we focus on proving that $g$ has a zero.

    Assume the contrary, that $g$ does not have a zero. Let us denote by 
    \[
        R \coloneqq \bigoplus_{i=1}^n \complexModule{i}^{d-i} \oplus \bigoplus_{1 \le i < j \le n}\complexModule{i,j}
    \]
    the $(S^1)^n$-representation in the codomain.
    Then, by the monotonicity of the Fadell--Husseini index, we get
    \begin{equation} \label{eq: monotonicity of FH index, general}
        \FHindex{(S^1)^n}{\Z}{S(R)} \subseteq \FHindexBig{(S^1)^n}{\Z}{\prod_{i=1}^nS(\complexModule{i}^d)}.
    \end{equation}
    On the one hand, we have 
    \begin{equation} \label{eq: FH index of product of spheres, general}
        \FHindexBig{(S^1)^n}{\Z}{\prod_{i=1}^nS(\complexModule{i}^d)} = \left(u_1^{d}, \dots, u_n^{d}\right) \in \Z[u_1, \dots, u_n].
    \end{equation}
    This follows by the product formula for the index \cite{Fadell:1988tm}*{Prop.~3.1}, since $S^1$-index of each component $S(\C^d)$ is 
    \[
        \FHindex{S^1}{\Z}{S(\C^d)} = \ker\big(H^*(\CP^{\infty};\Z) \longrightarrow H^*(\CP^{d-1};\Z)\big) = \ker\big(\Z[u] \relbar\joinrel\twoheadrightarrow \Z[u]/(u^d)\big) = (u^{d}) \subseteq \Z[u].
    \]
    
    On the other hand, we have
    \begin{equation} \label{eq: FH index of representation, general}
        \FHindex{(S^1)^n}{\Z}{S(R)} = \Big(u_1^{d-1} \cdots u_n^{d-n} \cdot \prod_{1 \le i < j \le n} (u_i-u_j)\Big) \subseteq \Z[u_1, \dots, u_n].
    \end{equation}
    Indeed, from the Gysin sequence it follows that the index is generated by the top Chern class (i.e., Euler class) of the representation bundle
    \[
        \rho\colon R \longrightarrow R \times_{(S^1)^n} E(S^1)^n \longrightarrow B(S^1)^n.
    \]
    Using the bundle notation \eqref{eq:xi_i bundle}, we have
    \[
        \rho ~\cong~ \bigoplus_{i=1}^n \xi_i \oplus \bigoplus_{1 \le i < j \le n} \xi_i \otimes_{\C} \overline{\xi}_j,
    \]
    so the total Chern class equals
    \[
        c(\rho) = \prod_{i=1}^nc(\xi_i)^{d-i}\prod_{1 \le i<j < n} c(\xi_i \otimes_\C \overline{\xi}_j) = \prod_{i=1}^n(1+u_i)^{d-i}\prod_{1 \le i<j < n} (1 + u_i - u_j) \in \Z[u_1, \dots, u_n]
    \]
    and the claim \eqref{eq: FH index of representation, general} holds.

    We argue now by contradiction.
    The term 
    \[
        u_1^{d-1} u_2^{d-2} \cdots u_n^{d-n} \cdot \prod_{1 \le i < j \le n} u_j = (u_1 u_2 \cdots u_n)^{d-1}
    \]
    is a summand with a sign $\pm 1$ in the generator of \eqref{eq: FH index of representation, general}, which does not get canceled by other summands, but cannot appear as a summand in an element of \eqref{eq: FH index of product of spheres, general}, contradicting \eqref{eq: monotonicity of FH index, general}.
\end{proof}

\begin{remark}[$\Z_2^n$-strengthening of Theorem \ref{thm:equiv}] \label{rem:Z_2^n}
    One can strengthen Theorem \ref{thm:equiv} by showing that every $\Z_2^n$-equivariant map 
    \[
        W_n(\C^d) \longrightarrow \R_1^{a_1}\oplus \dots \oplus \R_n^{a_n}
    \]
    hits the origin if $a_{i}+2(i-1) \le 2d-1$ for each $i = 1, \dots, n$. Here by $\R_i$ we mean the analogous one-dimensional representation of $\Z_2^n$ to $\C_i$. This is a strengthening of Theorem \ref{thm:equiv} by setting $a_i \coloneqq 2(d-i)$ for each $i = 1, \dots, n$ if we consider $\Z_2^n$ as a subgroup of $(S^1)^n$. 
    
    In order to prove this, by using the same extension idea, one reduces the task to showing that every $\Z_2^n$-map
    \[
        \prod_{i=1}^nS(\R_i^{2d}) \longrightarrow \bigoplus_{i=1}^n \R_i^{a_i} \oplus \bigoplus_{1 \le i < j \le n}\complexModule{i,j}
    \]
    hits the origin.
    Indeed, when Fadell--Husseini indices are considered with respect to group $\Z_2^n$ and $\F_2$ coefficients, we obtain a contradiction by monotonicity property in this case as well, since
    \[
        \FHindexBig{\Z_2^n}{\F_2}{\prod_{i=1}^n S(\R_i^{2d})} = (t_1^{2d}, \dots, t_n^{2d}) \subseteq \F_2[t_1, \dots, t_n]
    \]
    and
    \[
        \FHindexBig{\Z_2^n}{\F_2}{S\big(\bigoplus_{i=1}^n \R_i^{a_i} \oplus \bigoplus_{1 \le i < j \le n}\complexModule{i,j}\big)} = \Big(t_1^{a_1} \cdots t_n^{a_n} \cdot \prod_{1 \le i < j \le n} (t_j^2+t_i^2)\Big) \subseteq \F_2[t_1, \dots, t_n].
    \]
    One can obtain these in the same way as in the proof of Theorem \ref{thm:equiv}, after noting an isomorphism 
    \[
        \complexModule{i,j}= \complexModule{i} \otimes_\C \complexModule{j} \cong \R_i^{\oplus 2} \otimes_{\C} \R_j^{\oplus 2} \cong (\R_{i}\otimes \R_{j})^{\oplus 2}
    \]
    of real $\Z_2^n$-representation, which contributes a factor of $(t_i+t_j)^2=t_i^2+t_j^2$ in the generator of the latter index.
    
    Again, the contradiction is obtained by noting that a nonzero term
    \[
        t_1^{a_1} t_2^{a_2} \cdots t_n^{a_n} \cdot \prod_{1 \le i < j \le n} t_j^2 = t_1^{a_1} t_2^{a_2+2} \cdots t_n^{a_n+2(n-1)}
    \]
    is in the generator of the second ideal, but never appears as a summand in an element of the first ideal, due to $a_i+2(i-1) \le 2d-1$, for each $i = 1, \dots, n$.
\end{remark}

\section{Complex Tverberg--Vre\'cica theorems}
\label{sec:tverberg-vrecica}

The goal of this section is to prove Theorem \ref{thm:strong-Tv-vrecica-complex} and Theorem \ref{thm:strong-tverberg-vrecica-odd-flats}, as well as the colorful version Theorem \ref{thm:colored-tverberg-vrecica}. We do this by exhibiting a slightly more general Theorem \ref{thm:tverberg vrecica general} and applying the results on non-vanishing powers of certain Euler classes from Section \ref{sec:euler-classes}. A similar approach in the case of real flats were used by Karasev \cites{Karasev:2007ib} for the classical Tverberg-Vre\'cica conjecture and by Blagojevi\'c, Matschke and Ziegler \cite{blagojevic2011optimal} for the colored version. The configuration space -- test map scheme is a generalization of the deleted join approach \cite{matouvsek2003borsuk} for the topological Tverberg theorem, and variants thereof.\\

We first introduce some notation. We refer the reader also to Section \ref{sec:notation} for further conventions.
Given integers $0 \le k \le d$ and a closed subset $B \subseteq \grassmannR{d-k}{d}$, we will denote by $\gamma_B$ the restriction of the canonical bundle $\gamma_{d-k}^\R$ over the Grassmannian to $B$. If, for a prime number $p$, the bundle $\gamma_B$ is oriented mod $p$, we will denote by $e(\gamma_B) \in H^*(B; \F_p)$ the mod $p$ Euler class. 

For a vector space $U$ and an integer $r \ge 1$, let 
\[
    W_r(U) \coloneqq\{(v_1, \dots, v_r) \in U^{\oplus r} \colon~ v_1 + \dots + v_r = 0\}.
\]
This vector space admits an action of the symmetric group $\Sym_r$ by permuting the coordinates. We will be interested in the case when $r = p^{\alpha}$ is a power of a prime $p$. Then, we consider $\Z_p^{\alpha} \subseteq \Sym_{r}$ as a subgroup via the regular embedding \cite[Ex.~III.2.7]{adem2013cohomology}, and hence acting on $W_r(U)$.
Let
\begin{align*}
    \pi_r \colon U^{\oplus r} \longrightarrow W_r(U),~~(v_i)_{i=1}^r \longmapsto (v_i)_{i=1}^r - ((v_1+\dots +v_r)/r)_{i=1}^r
\end{align*}
denote the $\Sym_r$-equivariant projection.
Generalizing this, for a vector bundle $\xi$ with a fiber $U$,
we will denote by $W_r(\xi) \subseteq \xi^{\oplus r}$ the subbundle with a fiber $W_r(U)$. The symmetric group $\Sym_r$ acts fiberwise on $W_r(\xi)$ by permuting the coordinates.

For a simplicial complex $K$, let 
\[
    \deletedJoinNoBrackets{K}{r} \coloneqq \{\sigma_1* \dots * \sigma_r \colon~ \sigma_i \textrm{ are pairwise disjoint}\}
\]
denote the $r$-fold deleted join of $K$. The points (in its geometric realization) will be denoted by $t_1 x_1 \uplus \dots \uplus t_r x_r \in \sigma_1*\dots*\sigma_r \subseteq \deletedJoinNoBrackets{K}{r}$, for $x_i \in \sigma_i$ and $t_1, \dots, t_r \ge 0$ with $t_1+ \dots + t_r = 1$. Deleted join admits a free $\Sym_r$-action by permuting the coordinates. \\

We have the following general result. See also \cites{blagojevic2011optimal, Karasev:2007ib}.

\begin{theorem} \label{thm:tverberg vrecica general}
    Let $m \ge 0$ and $0 \le k \le d$ be integers, and let $p$ a prime number. Let $B \subseteq \grassmannR{d-k}{d}$ be a closed subset and $\gamma_B$ its canonical bundle, as described above. For each $i=0, \dots, m$, let $r_i=p^{\alpha_i}$ be a power of $p$, let $K_i$ be a simplicial complex with 
    \[
    	\FHindex{\Z_p^{\alpha_i}}{\F_p}{\deletedJoin{K_i}{r_i}} \subseteq H^{> (r_i-1)(d-k+1)}(B\Z_p^{\alpha_i};\F_p),
    \]
    and let $f_i \colon K_i \longrightarrow \R^d$ be a continuous map.
    If $\gamma_B$ is oriented mod $p$ and $e(\gamma_B)^m \neq 0 \in H^*(B; \F_p)$, then there exist $U \in B$ and $x\in U$ and, for each $i=0, \dots, m$, there exist points $x_1^i, \dots, x^i_{r_i} \in K_i$ having pairwise disjoint supports such that the $k$-flat $x+U^{\perp} \subseteq \R^d$ satisfies
    \[
        \bigcup_{i=0}^m\{f_i(x_1^i), \dots, f_i(x^i_{r_i})\} \subseteq x+U^{\perp}.
    \]
\end{theorem}
\begin{proof}
    Let $U \in B$ and let $\pi_U \colon \R^{d} \longrightarrow U$ denote the orthogonal projection. The goal is to show that some $x \in U$ is in a Tverberg partition of each of the maps $g_i \coloneqq \pi_U \circ f_i \colon K_i \longrightarrow U$, for $i=0,\dots, m$. Then, $x+U^{\perp}$ would be the desired $k$-flat. 
    
    For each $i=0, \dots, m$, let
    \begin{align*}
        \deletedJoin{g_i}{r_i} \colon \deletedJoin{K_i}{r_i} \longrightarrow W_{r_i}(U \oplus \R),
         ~~ t_1x_1 \uplus \dots \uplus t_{r_i}x_{r_i}\longmapsto \pi_{r_i}\big(t_1g_i(x_1), t_1, \dots, t_{r_i}g_i(x_{r_i}), t_{r_i}\big),
    \end{align*}
    be the $\Z_p^{\alpha_i}$-equivariant deleted join map, where we consider $\Z_p^{\alpha_i} \subseteq \Sym_{r_i}$ as a subgroup acting by permuting the coordinates. Using these, we may construct a $\prod_{i=0}^m \Z_p^{\alpha_i}$-equivariant test map
    \begin{align} \label{eq:test map on fiber}
        \prod_{i=0}^m \deletedJoin{K_i}{r_i} &\longrightarrow W_{m+1}(U) \oplus \bigoplus_{i=0}^m W_{r_i}(U \oplus \R)\\
        \Big(\biguplus_{j=1}^{r_i}t_j^ix_j^i \Big)_{i=0}^m & \longmapsto \pi_{m+1} \Big(\sum_{j=1}^{r_i} t_j^ig_i(x_j^i)\Big)_{i=0}^m \oplus \bigoplus_{i=0}^m \deletedJoin{g_i}{r_i}\Big(\biguplus_{j=1}^{r_i}t_j^ix_j^i \Big) \notag
    \end{align}
    which hits the origin if and only if the point $x \in U$ from the first paragraph exists.
    There is a natural fiberwise $\prod_{i=0}^m \Z_p^{\alpha_i}$-action on the bundle $W_{m+1}(\gamma_B) \oplus \bigoplus_{i=0}^m W_{r_i}(\gamma_B \oplus \varepsilon)$ over $B$, where $\varepsilon \colon~ \R \longrightarrow \R \times B \longrightarrow B$ denotes the trivial real line bundle. Multiplying both the total and the base space with $\prod_{i=0}^m \deletedJoin{K_i}{r_i}$, we get a vector bundle
    \begin{equation} \label{eq:intermediate-bundle}
    	W_{m+1}(\gamma_B) \oplus \bigoplus_{i=0}^m W_{r_i}(\gamma_B \oplus \varepsilon) \times \deletedJoin{K_i}{r_i}
    \end{equation}
    over the base space $B \times \prod_{i=0}^m \deletedJoin{K_i}{r_i}$ with a fiber $W_{m+1}(U) \oplus \bigoplus_{i=0}^m W_{r_i}(U \oplus \R)$.
    Collecting the test maps \eqref{eq:test map on fiber} for each $U \in B$, we get an $\prod_{i=0}^m \Z_p^{\alpha_i}$-equivariant section of \eqref{eq:intermediate-bundle}. Our problem boils down to showing that it has a zero.
    
    After performing the Borel construction on both the total and the base space with respect to the diagonal $\prod_{i=0}^m \Z_p^{\alpha_i}$-action, the problem reduces to showing that there are no nowhere zero sections of the vector bundle
    \[
        W_{m+1}(\gamma_B) \oplus \bigoplus_{i=0}^m \big(W_{r_i}(\gamma_B \oplus \varepsilon) \times \deletedJoin{K_i}{r_i}\big) \times_{\Z_p^{\alpha_i}} E\Z_p^{\alpha_i}
    \]
    over the base $B \times \prod_{i=0}^m \deletedJoin{K_i}{r_i} \times_{\Z_p^{\alpha_i}} E\Z_p^{\alpha_i}$. However, due to our assumptions, we can apply Lemma \ref{lem:euler class of general bundle} to deduce that this bundle has a nontrivial mod $p$ Euler class and hence admits no nowhere zero sections.
\end{proof}

We now recall the statements of the topological complex Tverberg--Vre\'cica results before proving them.

\thmTVeven*
\begin{proof}
    The statement follows by applying Theorem \ref{thm:tverberg vrecica general} to $B = \grassmannC{d-k}{d} \subseteq \grassmannR{2d-2k}{2d}$, $m = k$ and $K_i = \Delta_{N_i}$. Let us show that the conditions of theorem are satisfied.
    Indeed, $\gamma_B$ is the canonical bundle $\gamma_{d-k}^\C$ over $\grassmannC{d-k}{d}$, which is canonically oriented, and Lemma \ref{lem:euler class of C-grassmannian} implies 
    \[
        e(\gamma_B)^k = e(\gamma^\C_{d-k})^k \neq 0 \in H^*(\grassmannC{d-k}{d};\F_p).
    \]
    Since $\deletedJoin{\Delta_{N_i}}{r_i} \cong [r_i]^{*N_i+1}$ is $(N_i-1)$-connected, it follows from the Leray-Serre spectral sequence of the Borel fibration of $\deletedJoin{\Delta_{N_i}}{r_i}$ that 
    \[
    	\FHindex{\Z_p^{\alpha_i}}{\F_p}{\deletedJoin{\Delta_{N_i}}{r_i}} \subseteq H^{>N_i}(B\Z_p^{\alpha_i};\F_p) = H^{> (r_i-1)(\rank_\R(\gamma_B)+1)}(B\Z_p^{\alpha_i};\F_p),
    \]
    which finishes the proof.
\end{proof}

To every coloring of the vertex set of a simplex $\Delta_N$ one can associate the so-called \emph{rainbow complex}. This is a subcomplex of $\Delta_N$ consisting of all \emph{rainbow faces} of the simplex, that is, the faces which have at most one vertex of each color. 

We may now prove the topological generalization of Theorem \ref{thm:colored-tverberg-vrecica}. 

\begin{theorem}[Topological colored Tverberg-Vre\'cica for complex flats]
	Let $0 \le k < d$ be integers and $p$ be a prime number. For each $i=0, \dots, k$, let $N_i \coloneqq (r_i-1)(2d-2k+1)$ and suppose that the vertex set of $\Delta_{N_i}$ is colored such that each of the color classes are of size at most $p-1$. Suppose that $f_i \colon \Delta_{N_i} \longrightarrow \C^d$ are continuous maps.
    Then, for each $i=0, \dots, k$ there are points $x_1^i, \dots, x^i_{p} \in \Delta_{N_i}$ whose supports are pairwise disjoint rainbow faces, and there is an affine complex $k$-dimensional subspace $V \subseteq \C^d$ such that
    \[
        \bigcup_{i=0}^k\{f_i(x_1^i), \dots, f_i(x^i_{p})\} \subseteq V.
    \]
\end{theorem}
\begin{proof}
	The statement follows by applying Theorem \ref{thm:tverberg vrecica general} to $B = \grassmannC{d-k}{d} \subseteq \grassmannR{2d-2k}{2d}$, $m = k$ and $K_i \subseteq \Delta_{N_i}$ the rainbow complex. Let us check the conditions of the theorem. As in the previous proof, we have $e(\gamma_B)^k \neq 0 \in H^*(B;\F_p)$. If we denote by $V(\Delta_{N_i}) = \bigsqcup_j C_{i,j}$ the coloring of the vertex set of the $i$-th simplex, then $K_i$ is isomorphic to the join of chessboard complexes $\Delta_{r, |C_{i,j}|}$. By \cite[Cor.~2.6]{blagojevic2011optimal} and using the fact that $\sum_j |C_{i,j}| = N_i +1$, we have
    \[
    	\FHindex{\Z_p}{\F_p}{\deletedJoin{K_i}{p}} =  H^{\ge N_i+1}(B\Z_p;\F_p) = H^{\ge (r_i-1)(\rank_\R(\gamma_B)+1)+1}(B\Z_p;\F_p),
    \]
    which finishes the proof.
\end{proof}

We may now also prove the odd codimension result, where the flat is a Minkowski sum of a complex flat with a real affine line. Let us first recall the statement.

\thmTVodd*
\begin{proof}
    The statement follows by applying Theorem \ref{thm:tverberg vrecica general} to 
    \[
        B = \{(T, L) \in \grassmannC{d-k}{d} \times \RP^{2d-1}\colon~ T \perp L\} \subseteq \grassmannR{2d-2k+1}{2d},~(T,L) \longmapsto T \oplus L
    \]
    $m=k$ and $K_i = \Delta_{N_i}$.
    To see that $B$ is indeed a subset of $\grassmannR{2d-2k+1}{2d}$, we may note that $T$ is the maximal invariant subspace of $T \oplus L$ under multiplication by $i$ and $L$ the real-orthogonal complement of $T$ in $T \oplus L$.
    For $(T,L) \in B$,
    we recover $V$ from the statement as a translate of $(T\oplus L)^{\perp_\R} = (T \oplus L \oplus iL)^{\perp_\C} \oplus iL$, where $i$ denotes the unit imaginary complex number. 
    
    Let us check the conditions of the theorem we are using.
    The space $B$ can be identified with the real projectivisation $\PP_\R((\gamma^\C_{d-k})^{\perp})$ of the orthogonal complement of the canonical bundle over $\grassmannC{d-k}{d}$.
    If we denote by $\gamma_1^{\R}$ the canonical line bundle over $\RP^{2d-1}$, then
    \[
        \gamma_B \cong \pi_1^*(\gamma^\C_{d-k}) \oplus  \pi_2^* (\gamma_1^{\R}),
    \]
    where
    \[
        \pi_1 \colon \PP_\R((\gamma^\C_{d-k})^{\perp}) \longrightarrow \grassmannC{d-k}{d} \hspace{0.3cm} \textrm{and} \hspace{0.3cm}  \pi_2 \colon \PP_\R((\gamma^\C_{d-k})^{\perp}) \longrightarrow \RP^{2d-1}
    \]
    are induced maps from the two projections $(T,L) \longmapsto T$ and $(T, L) \longmapsto L$, respectively. Note that $\gamma_B$ is oriented mod 2 and by Corollary \ref{cor:euler class mod 2 of 2m + 1 grassmannian}, we have
    \[
        e(\gamma_B)^k = e\big(\pi_1^*(\gamma^\C_{d-k}) \oplus  \pi_2^*(\gamma_1^{\R})\big)^{k} \neq 0 \in H^*(\PP_\R((\gamma^\C_{d-k})^{\perp}); \F_2).
    \]
    Similarly as in the proof of the previous theorem, we have
    \[
    	\FHindex{\Z_p^{\alpha_i}}{\F_p}{\deletedJoin{\Delta_{N_i}}{r_i}} \subseteq H^{>N_i}(B\Z_p^{\alpha_i};\F_p) = H^{> (r_i-1)(\rank_\R(\gamma_B)+1)}(B\Z_p^{\alpha_i};\F_p),
    \]
    which finishes the proof.
\end{proof}

\section{Lemmas on Euler classes}\label{sec:euler-classes}

In this section we prove non-vanishing of the appropriate Euler classes used in Section \ref{sec:tverberg-vrecica}. The next two lemmas, although in a different shape, are implicit in \cite{Karasev:2007ib} and later appeared in a similar form in \cite{blagojevic2011optimal}.

Let us first introduce a bit of notation. The reader is referred  to Section \ref{sec:notation} and Section \ref{sec:tverberg-vrecica} for further conventions. Let $r=p^\alpha$ be a power of $p$. We will consider $\Z_p^\alpha \subseteq \Sym_r$ as a subgroup via the regular embedding \cite[Ex.~III.2.7]{adem2013cohomology}. For a simplicial complex $K$, a vector bundle $\xi$ over a paracompact base $B$ with a fiber $V$ and a trivial real line bundle $\varepsilon \colon \R \longrightarrow \R \times B \longrightarrow B$, let
\begin{equation} \label{eq:Wr x Delta bundle}
    (W_{r}(\xi \oplus \varepsilon) \times \deletedJoinNoBrackets{K}{r}) \times_{\Z_p^{\alpha}} E\Z_p^{\alpha}
\end{equation}
be a vector bundle over the base space $B \times (\deletedJoinNoBrackets{K}{r} \times_{\Z_p^{\alpha}} E\Z_p^{\alpha})$. Notice that it has a fiber $W_r(V \oplus \R)$ with the dimension $(r-1)(\rank_\R \xi+1)$.\\

We have the following auxiliary lemma.

\begin{lemma} \label{lem:euler class of Wr x Delta}
    Let $p$ be a prime number, $r = p^\alpha$, for some integer $\alpha \ge 0$, and $\xi$ a complex vector bundle over a paracompact base space $B$. Assume that $K$ is a simplicial complex with 
    \[
        \FHindex{\Z_p^{\alpha}}{\F_p}{\deletedJoinNoBrackets{K}{r}} \subseteq H^{> (r-1)(\rank_\R\xi + 1)}(B\Z_p^{\alpha};\F_p).
    \]
    Then, the mod $p$ Euler class of the vector bundle \eqref{eq:Wr x Delta bundle} has a nonzero term in
    \begin{equation*}
         H^*(\deletedJoinNoBrackets{K}{r} \times_{\Z_p^{\alpha}} E\Z_p^{\alpha};\F_p) ~\subseteq~ H^*(B;\F_p) \otimes H^*(\deletedJoinNoBrackets{K}{r} \times_{\Z_p^{\alpha}} E\Z_p^{\alpha};\F_p). 
    \end{equation*}
\end{lemma}
\begin{proof}
    Let $b \in B$ be a point and $V$ a fiber of $\xi$ over $b$. The claim is equivalent to showing that the pullback of the bundle \eqref{eq:Wr x Delta bundle} via the inclusion
    \[
        \deletedJoinNoBrackets{K}{r} \times_{\Z_p^{\alpha}} E\Z_p^{\alpha} ~ \xrightarrow{~\{b\} \times \textrm{id}} B \times (\deletedJoinNoBrackets{K}{r} \times_{\Z_p^{\alpha}} E\Z_p^{\alpha})
    \]
    has nonzero mod $p$ Euler class. However, this pullback is the top bundle in another pullback diagram
    \begin{equation} \label{eq:pullback of Wr bundles}
        \begin{tikzcd}
             W_r(V \oplus \R) \arrow[r] \arrow[d] & (\deletedJoinNoBrackets{K}{r} \times W_r(V \oplus \R)) \times_{\Z_p^{\alpha}} E\Z_p^{\alpha} \arrow[r] \arrow[d] & \deletedJoinNoBrackets{K}{r} \times_{\Z_p^{\alpha}} E\Z_p^{\alpha} \arrow[d, "\pi"]\\
            W_r(V \oplus \R) \arrow[r] & W_r(V \oplus \R) \times_{\Z_p^{\alpha}} E\Z_p^{\alpha} \arrow[r] & B\Z_p^{\alpha}
        \end{tikzcd}
    \end{equation}
    where $\pi$ is the projection in the Borel fibration of $\deletedJoinNoBrackets{K}{r}$. Let us denote by $\eta$ the bottom representation bundle.
    
    From the Fadell-Husseini index assumption, $\pi$ induces a monomorphism in cohomology up to degree $(r-1)(\rank_\R\xi + 1) = \dim W_r(V \oplus \R)$. Therefore, it is enough to prove that the mod $p$ Euler class $e(\eta) \in H^*(B\Z_p^\alpha;\F_p)$ is nonzero.

    The Gysin sequence \cite[Thm.~12.2]{MilnorStasheff74} for the spherical bundle
    \begin{equation*}
        S(W_r(V \oplus \R)) \longrightarrow S(W_r(V \oplus \R)) \times_{\Z_p^{\alpha}} E\Z_p^{\alpha} \xrightarrow{~~\pi_1~~} B\Z_p^{\alpha}
    \end{equation*}
    induced from $\eta$ has an exact portion
    \begin{equation*}
        \dots \to H^{n-(r-1)(\rank_{\R}\xi +1)}(B\Z_p^\alpha;\F_p) \xrightarrow{e(\eta)\cup} H^n(B\Z_p^\alpha;\F_p) \xrightarrow{\pi_1^*} H^n(S(W_r(V \oplus \R)) \times_{\Z_p^{\alpha}} E\Z_p^{\alpha};\F_p) \to \dots,
    \end{equation*}
    where the left arrow is multiplication by $e(\eta)$. Therefore, by exactness, the kernel of
    \[
        \pi_1^* \colon~ H^*(B\Z_p^\alpha;\F_p) \longrightarrow H^*(S(W_r(V \oplus \R)) \times_{\Z_p^{\alpha}} E\Z_p^{\alpha}; \F_p)
    \]
    is the principal ideal generated by $e(\eta)$.
    However, since the fix-point set $S(W_r(V \oplus \R))^{\Z_p^\alpha}$ is empty, it follows by a consequence of the localization theorem \cite[Cor.~I.1]{hsiang1975cohomology} that the kernel of $\pi_1^*$ is non-zero. Therefore, $e(\eta) \neq 0 \in H^*(B\Z_p^\alpha;\F_p)$, which finishes the proof.
\end{proof}

Extending the notation from before, let $m \ge 0$ be an integer, and for each $i=0, \dots, m$, let $r_i \coloneqq p^{\alpha_i}$ some power of $p$ and $K_i$ a simplicial complex.
Consider a map
\[
    B \times \prod_{i=0}^m \deletedJoin{K_i}{r_i} \times_{\Z_p^{\alpha_i}} E\Z_p^{\alpha_i} \xrightarrow{~~\Delta \times \textrm{id}} B \times \prod_{i=0}^m B \times (\deletedJoin{K_i}{r_i} \times_{\Z_p^{\alpha_i}} E\Z_p^{\alpha_i})
\]
defined as the product of the diagonal map $\Delta \colon B \longrightarrow B \times \prod_{i=0}^m B$ and the identity on other factors. We will denote by
\begin{equation} \label{eq:big W bundle}
        W_{m+1}(\xi) \oplus \bigoplus_{i=0}^m (W_{r_i}(\xi \oplus \varepsilon) \times \deletedJoin{K_i}{r_i}) \times_{\Z_p^{\alpha_i}} E\Z_p^{\alpha_i}
\end{equation}
the pullback of the product bundle
\begin{equation} \label{eq:product-big-bundle}
    W_{m+1}(\xi) \times \prod_{i=0}^m (W_{r_i}(\xi \oplus \varepsilon) \times \deletedJoin{K_i}{r_i}) \times_{\Z_p^{\alpha_i}} E\Z_p^{\alpha_i}
\end{equation}
along $\Delta \times \textrm{id}$. In particular, \eqref{eq:big W bundle} has $B \times \prod_{i=0}^m \deletedJoin{K_i}{r_i} \times_{\Z_p^{\alpha_i}} E\Z_p^{\alpha_i}$ as the base space.\\

Next, we prove a general result on non-vanishing of the Euler class used in Section \ref{sec:tverberg-vrecica}.

\begin{lemma} \label{lem:euler class of general bundle}
    Let $m \ge 0$ be an integer, $p$ a prime number and $\xi$ a vector bundle over a paracompact base space $B$. For each $i=0, \dots, m$, let $r_i = p^{\alpha_i}$ and let $K_i$ be a simplicial complex with 
    \[
    	\FHindex{\Z_p^{\alpha_i}}{\F_p}{\deletedJoin{K_i}{r_i}} \subseteq H^{> (r-1)(\rank_\R\xi + 1)}(B\Z_p^{\alpha_i};\F_p).
    \]
    If $\xi$ is oriented mod $p$ and $e(\xi)^m \neq 0 \in H^*(B;\F_p)$, then, the mod $p$ Euler class of \eqref{eq:big W bundle} is nonzero in cohomology
    \[
        H^*(B;\F_p) \otimes \bigotimes_{i=0}^m H^*(\deletedJoin{K_i}{r_i} \times_{\Z_p^{\alpha_i}} E\Z_p^{\alpha_i};\F_p)
    \]
    of the base.
\end{lemma}
\begin{proof}
    The bundle \eqref{eq:big W bundle} is the pullback of the product bundle \eqref{eq:product-big-bundle} along $\Delta \times \id$, which is a map inducing the cup product on cohomology of $B$ and is the identity on other components. Therefore, by multiplicativity of the Euler class, it is enough to show that the bundle $W_{m+1}(\xi)$ over $B$ has a non-trivial mod $p$ Euler class and that, for each $i = 0, \dots, m$, the bundle
    \[
        \big(W_{r_i}(\xi \oplus \varepsilon) \times \deletedJoin{K_i}{r_i}\big) \times_{\Z_p^{\alpha_i}} E\Z_p^{\alpha_i},
    \]
    with the base space $B \times \big(\deletedJoin{K_i}{r_i} \times_{\Z_p^{\alpha_i}} E\Z_p^{\alpha_i} \big)$,
    has the mod $p$ Euler class with a nontrivial term in the tensor product component
    \[
        H^*(\deletedJoin{K_i}{r_i} \times_{\Z_p^{\alpha_i}} E\Z_p^{\alpha_i};\F_p) \subseteq H^*(B;\F_p) \otimes H^*(\deletedJoin{K_i}{r_i} \times_{\Z_p^{\alpha_i}} E\Z_p^{\alpha_i};\F_p).
    \]
    The former is indeed true, since $W_{m+1}(\xi) \cong \xi^{\oplus m}$ and by our assumption
    \[
        e(W_{m+1}(\xi)) = e(\xi)^m \neq 0 \in H^*(B; \F_p).
    \]
    The latter is true by Lemma \ref{lem:euler class of Wr x Delta}, which finishes the proof.
\end{proof}

Our next lemma on the nonvanishing power of the top Chern class mod $p$ is essentially classically known \cite[Ch.~III]{hiller1982geometry}. However, we include a proof for the sake of completeness. Analogous statements for real (and oriented) Grassmannians were treated in \cites{Zivaljevic1999, Karasev:2007ib}.

\begin{lemma} \label{lem:euler class of C-grassmannian}
    Let $1 \le n \le d$ be integers and $\gamma_n^\C$ the tautological bundle of the Grassmannian $G_n(\C^d)$. Then, the mod $p$ reduction of the Euler class $e(\gamma_n^\C)$ satisfies
    \[
        e(\gamma_n^\C)^{d-n} \neq 0 \in H^*(G_n(\C^d);\F_p),
    \]
    for any prime number $p$.
\end{lemma}
\begin{proof}
    Since integral cohomology of $G_n(\C^d)$ appears only in even degrees and is free (see Section \ref{sec:notation}), by the universal coefficient theorem \cite[Thm.~5.5.10]{spanier1989algebraic}, we have 
    \[
        H^*(G_n(\C^d);\F_p) \cong H^*(G_n(\C^d);\Z)\otimes \F_p,
    \]
    and the claim follows if we show that 
    \[
        e(\gamma_n^\C)^{d-n} = c_n^{d-n} \in H^*(G_n(\C^d);\Z) \cong \Z[c_1, \dots, c_n]/I_n
    \]
    does not equal to another class multiplied by an integer distinct from $\pm 1$.
    To shorten the notation, we will assume integral cohomology coefficients for the rest of the proof.

    Consider the quotient map 
    \[
        q_n \colon W_n(\C^d)/U(1)^n \longrightarrow W_n(\C^d)/U(n) = \grassmannC{n}{d}
    \]
    obtained from the diagonal inclusion $U(1)^n \hookrightarrow U(n)$. We will observe how $c_n^{d-n}$ is mapped via pullback $q_n^*$. Namely, as a first step, let us show that
    \begin{equation} \label{eq:cohomology of flag}
        H^*(W_n(\C^d)/U(1)^n) \cong \Z[u_1, \dots, u_n]/J_n,
    \end{equation}
    with $|u_1| = \cdots = |u_n| = 2$ and $J_n$ an ideal such that $(u_1 \dots u_n)^{d-n}$ is not a non-invertible integer multiple of another class in \eqref{eq:cohomology of flag}. We will do so by induction on $n \ge 1$. 
    
    In the next step we will show that $q_n$ induces a map in cohomology which satisfies $q_n^*(c_n) = u_1 \dots u_n$. This would finish the proof, as then $c_n^{d-n}$ could not be a non-invertible integer multiple of another class since $(u_1 \cdots u_n)^{d-n}$ is not.

    For the base case $n=1$, we have $W_1(\C^d)/U(1) = \CP^{d-1}$ and 
    \[
        H^*(\CP^{d-1}) \cong \Z[u_1]/(u_1^d),
    \]
    where $|u_1|=2$. Thus, $u_1^{d-1}$ is the generator of the degree $2(d-1)$ cohomology.

    Assume $n \ge 1$ and that the claim holds for $n-1$. Let 
    \[
        \nu \coloneq q_{n-1}^*((\gamma_{n-1}^{\C})^{\perp})
    \]
    be the pullback of the orthogonal complement $(\gamma_{n-1}^{\C})^{\perp}$ of the canonical bundle $\gamma_{n-1}^{\C}$ over $\grassmannC{n-1}{d}$. Let $U(1)^{n-1} \subseteq U(1)^n$ be the inclusion on the first $n-1$ coordinates. The projection from an $n$-frame to an $(n-1)$-frame forgetting the $n$-th vector induces the projection map of the sphere bundle
    \begin{equation*}
        S(\nu) \colon~ S(\C^{d-n+1}) \longrightarrow W_n(\C^d)/U(1)^{n-1} \longrightarrow W_{n-1}(\C^d)/U(1)^{n-1}.
    \end{equation*}
    Dividing by the fiberwise action of $U(1)$ acting on the $n$-th vector in $W_n(\C^d)$, we get the complex-projectivized bundle 
    \begin{equation*}
        \PP_\C(\nu) \colon~ \CP^{d-n} \longrightarrow W_n(\C^d)/U(1)^{n} \longrightarrow W_{n-1}(\C^d)/U(1)^{n-1}.
    \end{equation*}
    As a consequence of the Leray--Hirsch theorem (see \cite[proof of Thm.~3.2, pg.~81]{hatcher2003vector}) it follows that the cohomology of the total space of the projectivization of $\nu$ is expressed as
    \begin{align*}
        &H^*(W_n(\C^{d})/U(1)^n)\\ 
        &\cong H^*(W_{n-1}(\C^{d})/U(1)^{n-1}) \otimes \Z[u_n]/(u_n^{d-n+1}-u_n^{d-n}c_1(\nu) + \dots + (-1)^{d-n+1}c_{d-n+1}(\nu)),
    \end{align*}
    where $|u_n| = 2$ and $c_i(\nu)$ is the $i$-th Chern class of $\nu$. In particular, we obtain \eqref{eq:cohomology of flag} by setting
    \[
        J_n \coloneqq (J_{n-1}, u_n^{d-n+1}-u_n^{d-n}c_1(\nu) + \dots + (-1)^{d-n+1}c_{d-n+1}(\nu)).
    \]
    It follows form the induction hypothesis that $(u_1 \cdots u_{n-1})^{d-n} \in \Z[u_1, \dots, u_{n-1}]/J_{n-1}$ is not a non-invertible integer multiple of another element. The same holds for $(u_1 \cdots u_{n})^{d-n} \in \Z[u_1, \dots, u_{n}]/J_{n}$, due to the fact that from the above formula follows that $H^*(W_n(\C^{d})/U(1)^n)$ is a free module over $H^*(W_{n-1}(\C^{d})/U(1)^{n-1})$ with basis $1, u_n, \dots, u_n^{d-n}$. This finishes the induction.

    To complete the proof, as was already explained, we need to prove that $q_n$ induces a map in cohomology which satisfies $q_n^*(c_n)=u_1\cdots u_n$.
    To see this, notice that $q_n$ fits in the diagram
    \begin{equation*}
        \begin{tikzcd}
            W_n(\C^d)/U(1)^n \arrow[r, hook] \arrow[d, "q_n"] & W_n(\C^{\infty})/U(1)^n = BU(1)^n \arrow[d, "r_n"] \\
            W_n(\C^d)/U(n) \arrow[r, hook] & W_n(\C^{\infty})/U(n)=BU(n),
        \end{tikzcd}
    \end{equation*}
    where the horizontal maps are the classifying maps induced by the inclusion $W_n(\C^d) \hookrightarrow W_n(\C^{\infty})$ and $r_n$ is the canonical projection induced from the inclusion $U(1)^n \subseteq U(n)$ (also defined in Section \ref{sec:notation}).
    We obtain the claim by applying the cohomology functor and noticing that in the diagram
    \begin{equation*}
        \begin{tikzcd}
            \Z[u_1, \dots, u_n]/J_n  & \Z[u_1, \dots, u_n] \arrow[l, twoheadrightarrow] \\
            \Z[c_1, \dots, c_n]/I_n \arrow[u, "q_n^*"]  & \Z[c_1, \dots, c_n], \arrow[l, twoheadrightarrow] \arrow[u, "r_n^*"]
        \end{tikzcd}
    \end{equation*}
    we have $r_n^*(c_n)=u_1 \cdots u_n$, as already observed in Section \ref{sec:notation}.
\end{proof}

For the orthogonal complement $(\gamma_n^\C)^{\perp}$ of $\gamma_n^\C$, let
\begin{equation*}
    \PP_\R((\gamma_n^\C)^{\perp}) = \{(V, \ell) \in \grassmannC{n}{d} \times \RP^{2d-1} \colon~ V \perp \ell\}
\end{equation*}
be its real projectivization. As a consequence of the Leray-Hirsch theorem \cite[proof of Thm.~3.1]{hatcher2003vector}, it follows that
\begin{equation*}
    H^*(\PP_\R((\gamma_n^\C)^{\perp});\F_2) \cong H^*(\grassmannC{n}{d};\F_2) \otimes \F_2[x]/\Big(\sum_{i=0}^{n-d}x^{2(n-d-i)} w_{2i}((\gamma_n^\C)^{\perp})\Big),
\end{equation*}
where $x$ is the first Stiefel-Whitney class of the line bundle associated to the twofold covering
\[
    S((\gamma_n^\C)^{\perp}) \longrightarrow S((\gamma_n^\C)^{\perp})/\Z_2 = \PP_\R((\gamma_n^\C)^{\perp}).
\]
Notice that odd Stiefel--Whitney classes of the complex bundle $(\gamma_n^\C)^{\perp}$ are trivial and the class $x$ is also the pullback of the cohomology generator along the projection 
\[
    \pi_2 \colon \PP_\R((\gamma_n^\C)^{\perp}) \longrightarrow \RP^{2d-1},~(V, \ell) \longmapsto \ell.
\]
If we denote by 
\[
    \pi_1 \colon \PP_\R((\gamma_n^\C)^{\perp}) \longrightarrow \grassmannC{n}{d},~(V, \ell) \longmapsto V
\]
the other projection and by $\gamma_1^{\R}$ the canonical line bundle over $\RP^{2d-1}$,  then
\begin{equation*}
    \pi_1^*(\gamma^\C_n) \oplus  \pi_2^* (\gamma_1^{\R}) = \{(V, \ell, z) \in \grassmannC{n}{d} \times \RP^{2d-1} \times \C^d \colon~ V \perp \ell,~ z \in V \oplus \ell\}
\end{equation*}
is the canonical bundle over $\PP_\R((\gamma_n^\C)^{\perp})$. We now state a result on non-vanishing power of its mod 2 Euler class.

\begin{corollary} \label{cor:euler class mod 2 of 2m + 1 grassmannian}
    Let $1 \le n < d$ be integers. Let $\gamma_n^\C$ and $\gamma_1^\R$ denote the canonical bundles over $\grassmannC{n}{d}$ and $\RP^{2d-1}$, respectively. Then, the mod 2 Euler class of the canonical bundle $\pi_1^*(\gamma_n^\C) \oplus  \pi_2^* (\gamma_1^{\R})$ over $\PP_\R((\gamma_n^\C)^{\perp})$ satisfies
    \[
        e(\pi_1^*(\gamma_n^\C) \oplus  \pi_2^* (\gamma_1^{\R}))^{d-n} \neq 0 \in H^*(\PP_\R((\gamma_n^\C)^{\perp}); \F_2).
    \]
\end{corollary}
\begin{proof}
    Recall that
    \[
        H^*(\grassmannC{n}{d};\F_2) \cong \F_2[w_2, w_4, \dots, w_{2n}]/I_n,
    \]
    for an ideal $I_n$ (see Section \ref{sec:notation}).
    By the multiplicativity property of the Euler class, we need to show
    \[
        e(\pi_1^*(\gamma_n^\C) \oplus  \pi_2^* (\gamma_1^{\R}))^{d-n} = w_{2n}^{d-n}x^{d-n} \neq 0 \in \F_2[x, w_2, w_4, \dots, w_{2n}]/(I_n, \sum_{i=0}^{n-d}x^{2(n-d-i)} w_{2i}((\gamma_n^\C)^{\perp})).
    \]
    If we assume the contrary, then
    \[
        w_{2n}^{d-n}x^{d-n} = \sum_{i \ge 0}x^i \iota_i  + q \sum_{i=0}^{n-d}x^{2(n-d-i)} w_{2i}((\gamma_n^\C)^{\perp}) \in \F_2[x, w_2, w_4, \dots, w_{2n}],
    \]
    for a polynomial $q \in \F_2[x, w_2, w_4, \dots, w_{2n}]$ and some $\iota_0, \iota_1, \dots \in I_n$, which are zero after a large enough index. 
    
    From Lemma \ref{lem:euler class of C-grassmannian} it follows that $w_{2n}^{d-n} \notin I_n$, hence $q \neq 0$, so we may write $q=q_sx^s + \dots + q_0$, for $q_0, \dots, q_s \in \F_2[w_2, \dots, w_{2n}]$ with $q_s \neq 0$. Without loss of generality, we may also assume $q_s \notin I_n$, because otherwise we could add it to the $\iota_i$'s. It now follows that the monomial $x^{2(n-d)+s}$ has a coefficient $q_s+\iota_{2(n-d)+s} \neq 0$ on the right hand side of the equation and zero on the left hand side, yielding a contradiction.
\end{proof}

\section*{Acknowledgment}

The authors would like to thank the anonymous referee for their comments, which improved the exposition of the manuscript.

\bibliographystyle{plain}
\bibliography{references.bib}

\end{document}